\documentclass[12pt]{article}
 \expandafter\let\csname equation*\endcsname\relax
 \expandafter\let\csname endequation*\endcsname\relax
\usepackage{bm}
\usepackage{enumerate}
\usepackage[top=1in, bottom=1in, left=1in, right=1in]{geometry}
\usepackage{mathrsfs,color}
\usepackage{amsfonts}\usepackage{multirow}
\usepackage{amssymb}
\usepackage{dsfont}
\usepackage{caption,comment}
\usepackage{blkarray}
\usepackage{amsmath}
\usepackage{float}
\usepackage{amssymb,amsthm}
\usepackage[toc,page]{appendix}
\usepackage{graphicx,afterpage}
\usepackage{epstopdf}
\usepackage{cancel}
\usepackage{mathdots}
\usepackage[pdftex,colorlinks=true]{hyperref}

\usepackage{todonotes}

\usepackage{mhequ}
\hypersetup{CJKbookmarks,%
bookmarksnumbered,%
colorlinks,%
linkcolor=black,%
citecolor=black,%
plainpages=false,%
pdfstartview=FitH}
\numberwithin{equation}{section}
\numberwithin{figure}{section}

\newcommand\tabcaption{\def\@captype{table}\caption}

\newtheorem{thm}{Theorem}[section]
\newtheorem{cor}[thm]{cor}
\newtheorem{lem}[thm]{Lemma}

\newtheorem{aspt}[thm]{aspt}
\newtheorem{rem}[thm]{Remark}

\newtheorem{theorem}{Theorem}[section]

\newtheorem{proposition}[theorem]{Proposition}

\newcommand{\Xhat}{\widehat{X}}

\newcommand{\Chat}{\widehat{C}}

\newcommand{\E}{\mathbb{E}}
\newcommand{\Prob}{\mathbb{P}}

\newcommand{\reals}{\mathbb{R}}

\newcommand{\Kalman}{\mathcal{K}}

\newcommand{\Rhat}{\widehat{R}}
\newcommand{\rmi}{\mathbf{i}}

\newcommand{\Khat}{\widehat{K}}

\newcommand{\ehat}{\hat{e}}
\newcommand{\bfP}{\mathbf{P}}
\newcommand{\calP}{\mathcal{P}}
\newcommand{\Index}{\mathcal{I}}

\newcommand{\Rtilde}{\widetilde{R}}
\newcommand{\etilde}{\tilde{e}}

\begin{document}
\title{Rigorous accuracy and robustness analysis for two-scale reduced random Kalman filters in high dimensions}
\author{Andrew J  Majda  and  Xin T Tong }
\date{\today}
\maketitle
\begin{abstract}
Contemporary data assimilation often involves millions of prediction variables. The classical Kalman filter is no longer computationally feasible in such a high dimensional context. This problem can often be resolved by exploiting the underlying  multiscale structure, applying the full Kalman filtering procedures only to the large scale variables, and estimating the small scale variables with proper statistical strategies, including multiplicative inflation, representation model error in the observations, and crude localization. The resulting two-scale reduced filters can have close to optimal numerical filtering skill based on previous numerical evidence. Yet, no rigorous explanation exists for this success, because these modifications create unavoidable bias and model error. This paper contributes to this issue by  establishing a new error analysis framework for  two different reduced random Kalman filters, valid independent of the large dimension. The first part of our results examines the fidelity of the covariance estimators, which is essential for accurate uncertainty quantification.  In a simplified setting, this is demonstrated by showing the true error covariance is dominated by its estimators. In general settings, the Mahalanobis error and its intrinsic dissipation can indicate covariance fidelity. The second part develops upper bounds for the covariance estimators by comparing with proper Kalman filters. Combining both results, the classical tools for Kalman filters can be used as a-priori performance criteria for the reduced filters. In applications, these criteria guarantee the reduced filters are robust, and accurate for small noise systems. They also shed light on how to tune the reduced filters for stochastic turbulence. 
\end{abstract}

\section{Introduction}
Data assimilation, the numerical prediction procedure for partially observed processes, has been a central problem for science and engineering for decades. In this new age of technology, the dimensions of filtering problems have grown exponentially, as a result of  the increasingly abundant observations and ever growing demand for prediction accuracy. In geophysical applications such as numerical weather forecasting, the dimensions are staggeringly high, often exceeding $d=10^6$ for  the prediction variables, and $q=10^4$ for the observations. In such a context, the well known Kalman filter is no longer computationally feasible. Its  direct implementation requires high dimensional matrices product and inversion, resulting a computation complexity of $O(d^2q)$, which far exceeds modern computing capability. 
\par
One important strategy for high dimensional filtering is dimension reduction. Many geophysical and engineering problems have intrinsic multiscale structures \cite{Maj03, Majda_Wang, GM13}, where the large scale variables have more uncertainty and of more prediction importance. In comparison, the  small scale variables are driven by strong dissipation and fast oscillation, their values are more predictable but of less  significance. Intuitively, one would like to apply the full filtering procedures for the large scale variables, while estimating the small scale variables with some simplified strategy. This paper investigates two such general strategies: estimate the small scale variables by their statistical equilibrium state, or use a constant  statistical state as prior in each filtering step for the small scale. The resulting two-scale reduced filters will be called the dynamically decoupled reduced Kalman filter (DRKF) and general reduced Kalman filter (RKF) respectively. These ideas have been applied earlier to stochastic turbulence, and known as the reduced Fourier domain Kalman filter (RFDKF) and variance strong damping approximate filter (VSDAF), see chapters 3 and 7 of \cite{MH12}. Numerous numerical tests on these reduced filters \cite{MH12, BM14} have shown  their performances are close to optimal in various regimes. And because only the large scale variable of dimension $p$ is fully filtered, the complexity is reduced significantly to $O(dq^2+p^2 q)$. 
\par
While the two-scale reduced filters have  simple intuition and successful applications, there is no rigorous analysis framework for its performance. Precisely speaking, we are interested in the statistical and dynamical features of filter error $e_n$. In the classical Kalman filtering context, we have complete knowledge of $e_n$, as its  covariance is correctly estimated by the optimal filter,  which follows a Riccati equation that quickly converges to an equilibrium state \cite{Bou93}. As for the reduced filters, the filter error covariance $\E e_n\otimes e_n$ no longer matches its reduced estimator $C_n$ because of unavoidable model errors, which create bias through multiplicative inflation, representation error in the observations, and crude localization. Instead, it follows an online recursion where model reduction procedures constantly introduce structural biases. As a consequence, there is an intrinsic barrier between the reduced filters and the optimal one \cite{MB12,BM14}. The classical framework of showing approximate filters are close to the optimal one is not valid in this scenario \cite{Del96, CD02, mandel2011convergence, LTT14}.
\par 
This paper proposes and applies a new performance analysis strategy for the reduced filters in the subtle context of Kalman filters with random coefficients (\cite{Bou93}, and chapter 8 of \cite{MH12} for an application in large dimensions). It consists of two parts. The first part examines the fidelity of the reduced covariance estimator $C_n$, and aims to show  the true error covariance is not underestimated, which is essential for rigorous uncertainty quantification. The direct approach, showing $\E e_n\otimes e_n\preceq C_n$, is applicable to RKF if the dimension reduction procedure preserves this inequality, while the system noises are uncorrelated with the system coefficients. Another more general but weaker approach considers the Mahalanobis error $\|e_n\|^2_{C_n}=e_n^T C_n^{-1}e_n$. By showing $\frac{1}{d}\E\|e_n\|^2_{C_n}$ is bounded by a dimension free constant, we show the  error covariance estimator is not far off from the true value. This is carried out by the Mahalanobis error dissipation, which is an intrinsic dynamical mechanism for Kalman type updates. It holds for both RKF and DRFK even with system noises that are correlated. 
\par 
The second objective is to find a bound for the covariance estimator $C_n$. Two signal observation systems with augmented coefficients are considered, and we show their Kalman filter covariances  are respectively the covariance estimator of DRKF and an upper bound for the covariance estimator of RKF. By building this connection, we transfer our original problem of reduced filters to a problem of standard classical Kalman filters. The latter has a rich literature we can rely on, so there are multiple ways to bound $C_n$. In addition, we can rely on conditions of these augmented Kalman filters to ensure the dimension reduction procedures do not decrease the covariance. 
\par
 In combination, the previous results can also provide accuracy measurement for the reduced filters, in terms of the mean square error (MSE) $\E|e_n|^2$, and how far off they are from the optimal filter. In practice, many models like stochastic turbulence could have various ways to do the two-scale separation \cite{GLM14,LM15}. Moreover, most practical reduced filters employ various covariance inflation techniques to ensure no  covariance underestimation \cite{And07, And08, LDN08}. Which  dimension reduction method is better, and how to tune the filter parameters,  are important practical questions, yet previously can only rely on extensive numerical experiments for answers. In this perspective, our framework can be used for a priori answers, or rigorous support for previous numerical findings. 
\par
The remainder of this section intends to give a quick overview of our results, while the detailed formal statements along with the proofs are left in the later sections.

\subsection{Kalman filtering in high dimension}
Consider a signal-observation system with random coefficients \cite{Bou93}:
\begin{equation}
\label{sys:random}
\begin{gathered}
X_{n+1}=A_n X_n+B_n+\xi_{n+1}\\
Y_{n+1}=H_n X_{n+1}+\zeta_{n+1}
\end{gathered}
\end{equation}
where $\xi_{n+1}$ and $\zeta_{n+1}$ are two sequences of independent Gaussian noise,  $\xi_{n+1}\sim \mathcal{N}(0, \Sigma_n)$ and $\zeta_{n+1}\sim \mathcal{N}(0,\sigma_n)$. We assume the signal variable $X_n$ is of dimension $d$, the observation variable $Y_n$ is of dimension $q\leq d$, and the observation noise matrix $\sigma_n$ is nonsingular to avoid ill-posed problems.  The realizations of the dynamical coefficients $(A_n, B_n,\Sigma_n)$, the observation coefficients $(H_n, \sigma_n)$, as long as $Y_n$ are assumed to be available, and the objective is to estimate $X_n$. By considering general random coefficients, many interesting models involves intermittent dynamical regimes or observations can be included in our framework. Details will be discussed in Section \ref{sec:example}. 

The optimal filter for system \eqref{sys:random} is the Kalman filter \cite{Bou93, LS01, MH12}, assuming $(X_0,Y_0)$ is Gaussian distributed. It estimates $X_{n}$ with a Gaussian distribution $\mathcal{N}(m_n, R_n)$, where the mean and covariance follow a well known recursion:
\begin{equation}
\label{sys:optimal}
\begin{gathered}
m_{n+1}=A_n m_n+B_n+K_{n+1}(Y_{n+1}-H_n m_n),\quad R_{n+1}=\Kalman(\Rhat_{n+1}),\\
\Rhat_{n+1}=A_n R_nA_n^T+\Sigma_n, \quad K_{n+1}=\Rhat_{n+1} H_n(\sigma_n +H_n \Rhat_{n+1} H^T_n )^{-1},\\
\quad \Kalman(C)=C-CH_n(\sigma_n+H_n CH_n^T)^{-1}H_n^TC.\\
\end{gathered}
\end{equation}

The Kalman filter has found a wide range of applications in various fields. This is  due to its theoretical optimality, robustness and stability in the classical low dimensional setting. However, in many modern day applications where the system dimension reaches $10^6$, direct application of \eqref{sys:optimal} is no longer feasible, because the computation complexity of \eqref{sys:optimal} is roughly $O(d^2q)=10^{16}$, which is far beyond the speed of standard high performance computing, $10^{12}$. A simple complexity analysis with details is in Section \ref{sec:complexity}.  

Beside the dimension reduction strategies discussed below, there are various ways to approximate the Kalman filter by random sampling. These methods are known as the ensemble Kalman filters (EnKF) \cite{evensen04, And01}. They require various ad-hoc tuning techniques \cite{And07, And08, LDN08, LKMD09}. Moreover, most of  the theoretical properties of EnKF are not well understood, except recent results on well-posedness, nonlinear  stability, and geometric ergodicity \cite{KLS14,TMK15non, KMT15, TMK15}. Quantitative analysis of the filter error size remains an open question except in the limit of a large sample size that exceeds the dimension \cite{mandel2011convergence, LTT14}. 

\subsection{Two-scale separation and reduced Kalman filters}
Dynamical features of  system \eqref{sys:random} can often be exploited for dimension reduction and  fast computation of Kalman filters.  In this paper, we focus on scenarios where system \eqref{sys:random} has a two-scale separation
\[
X_n=\begin{bmatrix}X_n^L\\ X_n^S\end{bmatrix}.  
\] 
Here, $X_n^L$ consists of  $p(\ll d)$ large scale variables,  and $X_n^S$ consists of $d-p$ small scale variables. Throughout, $\calP_L$ and $\calP_S$ will denote the associated subspaces, and $\bfP_L$ and $\bfP_S$ will denoted the associated projections. 
\par
In Section \ref{sec:turbreduce}, we will consider a simple stochastic turbulence model (see \cite{MH12}), where $A_n$ is a  constant  matrix consists of  $2\times 2$ diagonal sub-blocks with  spectral norm  $\exp(-\nu h|k|^2)$. For the small scale Fourier modes  with large wavenumber $|k|$, $A_n$ is a very strong damping. As a consequence,  the small scale variables have very little uncertainty, and often are driven by fast oscillations. Their exact values are of little importance, and also stiff for numerical computations. This simple example captures a feature shared by many complicated turbulence models. In such a scenario, it is a common strategy to apply dimension reduction and try to filter only for the large scale part. 
\par
One naive way of dimension reduction would be directly ignoring the the small scale part. But this is usually problematic. Despite that  $X_n^S$ has small uncertainty in each coordinate, the observation operator $H_n=[H_n^L, H_n^S]$ involves all coordinates:
\begin{equation}
\label{tmp:twoscaleobs}
Y_{n+1}=H_n^L X_n^L+H_n^S X_n^S+\zeta_{n+1}.
\end{equation}
As a sum,  $X_n^S$ could have significant contribution to the observation $Y_n$. Directly ignoring the small scale would create a huge bias, as the filter would try to interpret the contribution of $X_n^S$ in term of $X_n^L$, which is called representation error \cite{GLM14, LM15}. The correct filter reduction requires some simple but educated estimation of the filter impact from the small scale variables. 
\subsubsection{Dynamical decoupled reduced Kalman filter (DRKF)}
One simple closure of the small scale variables would be their statistical equilibrium states. This idea was applied for  stochastic turbulence in chapter 7 of \cite{MH12} and named the RFDKF. To generalize it, we consider a simplified setting where the dynamics of the signal variable $X_n$ is decoupled between the two scales. In other words, the system coefficients of \eqref{sys:random} have the following block structure:
\begin{equation}
\label{eqn:blockdiag}
A_n=\begin{bmatrix}
A_n^L &0\\
0 & A_n^S
\end{bmatrix},\quad 
B_n=\begin{bmatrix}
B_n^L \\
B_n^S
\end{bmatrix},\quad
\xi_n=\begin{bmatrix}
\xi_n^L \\
\xi_n^S
\end{bmatrix},\quad
\Sigma_n=\begin{bmatrix}
\Sigma_n^L &0\\
0 & \Sigma_n^S
\end{bmatrix}.
\end{equation}
The diagonal $A_n$ used for stochastic turbulence in Section \ref{sec:turbreduce} obviously fits this description. Notice that with observation mixing the two scales \eqref{tmp:twoscaleobs}, the optimal Kalman filter does not necessarily have a block diagonal structure, so we cannot directly apply a large scale projection to \eqref{sys:optimal}. 

The DRKF filtering strategy comes as a combination of two ideas. First, if the small scale has very small fluctuation, then its mean conditioned on the system coefficients, $\mu^S_n$, will be a good estimator. Second, the small scale observation $H^S_n X^S_{n+1}$ is interpreted as a noisy perturbation to the large scale observation. We can remove the mean of this perturbation by letting 
\[
Y^L_n=Y_n-H_n^S\mu^S_{n+1}.
\]
We also need to consider the fluctuation at the small scale $\Delta X^S_{n+1}=X^S_{n+1}-\mu^S_{n+1}$. By interpreting it as a mean zero Gaussian noise, we need to include the representative error covariance:
\[
\sigma_n^L=\sigma_n+H_n^SV^S_{n+1}(H_n^S)^T.
\]
Here $V^S_n$ is the unfiltered covariance of $X_n^S$ conditioned on the system coefficients.  In this way, we treat $\{\Delta X^S_{n+1}\}$ as an independent sequence. Unfortunately this is not the case in reality and creates model error, and we remedy it by inflating the covariance in the end with a factor $r>1$. 

In summary, DRKF estimates $X^L_n$ and $X^S_n$ by Gaussian distributions $(\mu^L_n, C^S_n)$ and $(\mu^S_n, V^S_n)$ respectively. The mean and covariance sequences are updated as below:
\begin{equation}
\label{sys:DRKF}
\begin{gathered}
\mu^L_{n+1}=A_n^L \mu^L_n+B_n^L+K^L_{n+1}(Y^L_n -H_n^L(A_n^L \mu^L_n+B_n^L)),\\
C_{n+1}^L=r\Kalman_L(\Chat^L_{n+1}), \quad\Chat^L_{n+1}=A^L_n C^L_n (A^L_n)^T+\Sigma^{L}_n, \\
K^L_{n+1}=\Chat^L_{n+1}(H_n^{L})^T(\sigma^L_n+H_n^L \Chat^L_{n+1} (H^L_n)^T)^{-1},\\
\Kalman_L(\Chat^L_{n+1})=\Chat^L_{n+1}-\Chat^L_{n+1}(H_n^L)^T(\sigma^L_n+H_n^L \Chat_{n+1}^L(H_n^L)^T)^{-1} H_n^L \Chat^L_{n+1},\\
Y^L_n=Y_n-H^S_n \mu^S_{n+1},\quad \sigma_n^L=\sigma_n+H_n^SV^S_{n+1} (H_n^S)^T,\\
\mu^S_{n+1}=A_n^S\mu^S_n+B_n^S, \quad V^S_{n+1}=A_n^S C_{n}^S (A_n^S)^T+\Sigma_n^S.\\ 
\end{gathered}
\end{equation}
DRKF uses an idea like 3DVar on the small scales \cite{BLSZ13, LSS14} with reduced filtering of the large scales.
Since the filter essentially works only in the large scale subspace, the computational complexity is reduced to $O(q^3+p^2q)$ or $O(q^3+p^2q+d^2)$ , see Section \ref{sec:complexity}.  Also see chapter 8 of \cite{MH12} for an application of DRKF to random filtering of geophysical turbulence. 

\subsubsection{General RKF}
When the two scales are not dynamically decoupled, the DRKF \eqref{sys:DRKF} may have bad performances. This is because DRKF does not filter the small scale part, while the small scale error feeds back to the large scale estimation through the cross scale dynamics (see page 43 of \cite{MH12} for an example). Another more appropriate reduced filtering strategy would be  filtering the small scale with  a constant prior covariance $D_S$ as an estimate of the small scale dynamics. This will be called  a general reduced Kalman filter (RKF). It has been applied to stochastic turbulence in chapter 7 of \cite{MH12} and called VSDAF. 

To be specific, a fixed $\calP_S\otimes \calP_S$ matrix $D_S$ will be used as the prior for the small scale variables. So given a covariance estimator $C_n$ for $X^L_n$, the effective covariance of $X_n$ will be 
\[
C^+_n:=C_n+D_S. 
\]
In many applications, $D_S$ can be chosen as a multiple of the unfiltered equilibrium covariance of $X_n^S$. But it can also take other general matrix values. In summary, the RKF estimates $X_n$ by  a Gaussian distribution $\mathcal{N}(\mu_n, C_n+D_S)$, with the mean and covariance generated by a recursion:
\begin{equation}
\label{sys:RKF}
\begin{gathered}
\mu_{n+1}=A_n \mu_n+B_n+\Khat_{n+1}(Y_{n+1}-H_n \mu_n), \\
\Chat_{n+1}=A_n C^+_nA_n^T+\Sigma_n,\quad \Khat_{n+1}=\Chat_{n+1} H_n(\sigma_n +H_n \Chat_{n+1} H^T_n )^{-1},\\
C_{n+1}=r\bfP_L\Kalman(\Chat_{n+1})\bfP_L. 
\end{gathered}
\end{equation}
With a complexity estimation in Section \ref{sec:complexity}, we see RKF reduces the complexity to $O(d^2+dq^2+dp^2)$.  

Unlike DRKF, RKF applies a large scale covariance projection in the final step. This ensures the prior covariance for small scale variable at the next step is still $D_S$. Its practical effect is similar to the localization techniques that are widely applied, as both simplify the covariance structures. On the other hand, this projection may underestimate the error covariance for the new update. To offset this effect, a multiplicative inflation with $r>0$ is applied, and in the effective covariance estimator we also include the constant covariance $D_S$. Ideally, such inflations will remedy the possible covariance underestimation, so that
\begin{equation}
\label{eqn:ideal}
\Kalman(\Chat_{n+1})\preceq C^+_{n+1}=r\bfP_L \Kalman(\Chat_n)\bfP_L +D_S. 
\end{equation}
To be pragmatic, \eqref{eqn:ideal} holds only for large $n$, and we need to introduce a time series for the ratio between both sides. This will be formalized as Assumption \ref{aspt:working} in Section \ref{sec:fidelity}. Note that RKF requires much less detailed dynamics of the small scale than DRKF but still includes an estimate of the effect of the small scale on the observations.

\subsection{Covariance fidelity}
\label{sec:mahaintro}
Just like many other practical filters, although the reduced Kalman filters produce good estimates in various numerical tests, there is no good rigorous explanation of their successes. A quantitative analysis for the filter error is required for this purpose. In our context, the filter error of RKF and DRKF are given respectively  by
\[
e_n=X_n-\mu_n,\quad e^L_n=X_n^L-\mu_n^L.
\]
Notice that we do not consider the small scale estimator error for DRKF, as the small scale variables are not filtered there. Error analysis  for  reduced filters is much more difficult than the error analysis for the optimal filter. For the optimal filter \eqref{sys:optimal}, the covariance of the error $X_n-m_n$ is simply its estimator $R_n$, which can be easily studied by the associated Riccati equation \cite{Bou93}. For the reduced filters, the reduced estimators $C_n^+$ and $C_n^L$ clearly do not match the real filter error covariance, while the estimator is also biased by the dimension reduction.  
\par 
One classic error analysis strategy for non-optimal filters is to compare them with the optimal filter and show the differences are small \cite{Del96, CD02, mandel2011convergence, LTT14}. Roughly speaking, this strategy assumes the non-optimal filter is very close to the optimal filter at one time, and then exploits the intrinsic ergodicity and continuity of the optimal filter to show the difference remains small there after. Unfortunately this strategy is invalid for our reduced filters, because they are structurally different from the optimal filter \eqref{sys:optimal}. Evidently, $R_n$ may not have a block diagonal structure like $C^+_n$ does, it may not have its $\calP_S\otimes \calP_S$ sub-block being exactly $D_S$, and this sub-block can never be zero as in the case for DRKF. This is also known as the information barrier for reduced filters, investigated by \cite{MB12, BM14}. 
\par
A more pragmatic strategy would be looking for  intrinsic error statistical relations. In particular, it is important to check whether the reduced covariance estimators dominate the real error covariance, as  underestimating error covariance often causes  severe filter divergence (see chapter 2 of \cite{MH12}). The direct way will be looking for $\E e_n\otimes e_n\preceq C^+_n$. This is applicable for RKF if the system noises are independent of the system coefficients, for example when the latter are deterministic. But for general scenarios and DRKF, the error covariance matrix $\E e_n\otimes e_n$ is hard to track, as nonindependent system noises are involved in the recursion. For these difficult situations, we need to look at other weaker scalar statistics. 
\par
One natural choice would be the mean square error (MSE), $\E |e_n|^2$. But MSE works best when the error is isotropic, in other words the error has  equal strength in all directions. Our two-scale setting clearly does not fit into this description, as  the small scale error is much weaker.  In comparison, the Mahalanobis norm is a better error measurement.  Given a nonsingular $d\times d$ positive definite (PD) matrix $C$, it generates  a Mahalanobis norm on $\reals^d$:
\begin{equation}
\label{eqn:maha}
\|v\|_C^2:=v^T[C]^{-1}v. 
\end{equation}
This norm is central in many Bayesian inverse problems. For example, given the prior distribution of $X$ as $\mathcal{N}(b, C)$, and a linear observation $Y=HX+\xi$ with Gaussian noise $\xi\sim \mathcal{N}(0,\Sigma)$, the optimal estimate is the minimizer of $\|x-b\|^2_C+\|Y-Hx\|^2_\Sigma. $
In our context, it is natural to look at the non-dimensionalized Mahalanobis error $\frac{1}{d}\|e_n\|^2_{C_n^+}$ and $\frac{1}{p}\|e_n^L\|^2_{C_n^L}$. Based on our RKF formulation, the true state is estimated by $\mathcal{N}(\mu_n, C_n^+)$. A natural statistics that verifies this hypothesis is simply $\frac{1}{d}\E \|e_n\|^2_{C_n^+}$. If the hypothesis holds, this statics should roughly be of constant value. Comparing with the MSE, the Mahalanobis error discriminates directions, and penalizes errors in the small scale. Moreover, by showing the Mahalanobis error is bounded, we also show the error covariance estimate $C_n^+$ more or less captures the real error covariance. 
\par
The Mahalanobis error also has surprisingly good dynamical properties. In short, $\|e_n\|^2_{C_n^+}$ is a dissipative (also called exponentially stable) sequence. This is actually carried by an intrinsic inequality induced by the Kalman covariance update operator $\Kalman$. It was exploited by previous works in the literature \cite{Bou93, RGYU99} to show robustness of Kalman filters and extended Kalman filters (although the name Mahalanobis error is not explicitly used, but readers can identify it easily in the proofs). One major result of this paper is informally stated as below:

\begin{theorem}
\label{thm:mahainform}
When applying DRKF \eqref{sys:DRKF} to a  dynamically decoupled system \eqref{eqn:blockdiag},  the non-dimensionalized Mahalanobis filter error $\frac{1}{p}\E \|e^L_n\|_{C^L_n}^2$  decays exponentially fast and is eventually bounded by a dimension free constant. 

When applying RKF \eqref{sys:RKF} to a general system described by \eqref{sys:random}, if the large scale projection does not decrease covariance estimate so \eqref{eqn:ideal} holds,  the non-dimensionalized Mahalanobis filter error $\frac{1}{d}\E \|e_n\|_{C^+_n}^2$ decays exponentially fast and is eventually bounded by a dimension free constant. In addition, if the system noises are independent of all system coefficients, the second moment of error is dominated by its estimator: $\E e_n\otimes e_n\preceq C_n^+$. 
\end{theorem}
The formal description is given by Theorems \ref{thm:deterministic}, \ref{thm:dissmaha}, and \ref{thm:errorDRKF}. The requirement of \eqref{eqn:ideal} will be replaced by a concrete version Assumption \ref{aspt:working}, which is chosen to be always valid in Section \ref{sec:offline} for RKF provided that $\bfP_S \Rtilde \bfP_S\preceq CD_S$ for a suitable constant $C>0$ depending on $r$, where $\Rtilde$ is the stationary asymptotic covariance for \eqref{sys:RKF}. 

\subsection{Intrinsic filter performance criteria}
\label{sec:referenceintro}
 Theorem \ref{thm:mahainform} essentially shows that the Mahalanobis error is a natural and convenient statistics to assess reduced filter performance. On the other hand it raises two new questions for us to address:
 \begin{itemize}
\item Bounds of the Mahalanobis error are informative  only if the covariance estimator $C_n^L$ or $C_n^+$ is bounded. So how can these estimators be bounded?
\item A large scale projection is applied for RKF in the assimilation step. It may decrease covariance estimation. Ideally this can be offset by the covariance inflations so \eqref{eqn:ideal} holds. In principle, \eqref{eqn:ideal} requires online verifications during the implementation of RKF. Yet, offline a priori criteria that depend only on the system coefficients are more desirable.
 \end{itemize}
Moreover, since the long term performance is more useful, the answers to the previous questions should not depend on the filter initialization. 
\par 
Let us consider DRKF first, which requires answering only the first question. In fact, the answer is quite straightforward. Consider the following augmented signal-observation system: 
\begin{equation}
\label{sys:inflatedDRKF}
\begin{gathered}
X^{L'}_{n+1}=A'_n X'_n+B_n+\xi^L_{n+1},\quad
Y'_{n+1}=H^L_n X^{L'}_{n+1}+\zeta'_{n+1},\\
A'_n=\sqrt{r}A^L_n,\quad\xi^L_{n+1}\sim \mathcal{N}(0, \Sigma'_n),\quad \Sigma'_n=r'\Sigma^L_n,\quad \zeta_{n+1}\sim \mathcal{N}(0,\sigma_n^L).
\end{gathered}
\end{equation}
The optimal filter of the above system is a Kalman filter $\mathcal{N}(m_n^L, R^L_n)$. It is easy to verify that $R^L_n=C^L_n/r$ if it holds at $n=0$, because $R^L_n$ follows a Riccati recursion just like \eqref{sys:DRKF}. The advantage we gain from this observation is that, as a Kalman filter covariance, $R^L_n$ converges to a unique stationary solution $\Rtilde^L_n$, assuming the system \eqref{sys:inflatedDRKF} is stationary, ergodic, weakly observable and controllable (See \cite{Bou93} and Theorem \ref{thm:Bou93}). This stationary solution reflects the intrinsic filtering skills of \eqref{sys:inflatedDRKF}. It is clearly bounded and independent of the filter initialization, and in many cases it can be computed or admits simple concrete upper bounds.

The same idea holds similarly for RKF. The corresponding inflated signal-observation system is slightly different from \eqref{sys:random} with an inflation $r'>r$:
\begin{equation}
\label{sys:inflated}
\begin{gathered}
X'_{n+1}=A'_n X'_n+B_n+\xi'_{n+1},\quad
Y'_{n+1}=H_n X'_{n+1}+\zeta'_{n+1},\\
A'_n=\sqrt{r'}A_n,\quad\xi_{n+1}\sim \mathcal{N}(0, \Sigma'_n),\quad \Sigma'_n=r'\Sigma_n+r' A_nD_S A_n^T,\quad \zeta_{n+1}\sim \mathcal{N}(0,\sigma_n).
\end{gathered}
\end{equation}
If $\Rtilde_n$ denotes the associated stationary Kalman covariance sequence, then it possesses all the theoretical and computational advantages mentioned for $\Rtilde^L_n$. Theorem \ref{thm:cutoff} in below transfers these advantages to RKF by showing that $C_n^+\preceq r\Rtilde_n+D_S$. Moreover, Theorem \ref{thm:cutoff} shows that the online condition \eqref{eqn:ideal} can be verified by a similar version for the stationary solution $\Rtilde_n$, which will be Assumption \ref{aspt:cutoff}; but because $\Rtilde_n$ can be estimated a priori, we find an a priori criterion that guarantees the performance of RKF. 

In Section \ref{sec:example}, we will discuss some scenarios when $\Rtilde^L_n$ and $\Rtilde_n$ can be bounded explicitly in spectral norm or with respect to the optimal covariance $R_n$. Then Theorem \ref{thm:mahainform} implies the MSE $\E \|e_n\|^2$ is bounded, or the reduced filter performance is comparable with the optimal one. In many practical scenarios where the observation is frequent, the system noise $\Sigma_n$ and observation noise $\sigma_n$  are of scale $\epsilon$ comparing to other system coefficients. Then it is easy to verify in such a setting $\Rtilde_n$ and $\Rtilde^L_n$ scale like $\epsilon^2$, and so will the reduced filter errors. This is a nontrivial property for the reduced filters and evidently  very useful in practice. This is usually framed as the accuracy of system estimators \cite{BLSZ13, LSS14}. 

Other than accuracy, another important application for our framework is finding the transition point for two-scale separation, and how to setup the small scale covariance $D_S$ for RKF. These questions can be answered by studying the Kalman filters for \eqref{sys:inflatedDRKF} and \eqref{sys:inflated}. Section \ref{sec:example} discusses these issues with  concrete examples  in stochastic turbulence.

\subsection{Preliminaries}
The remainder of this paper is arranged as follows. Although RKF \eqref{sys:RKF} applies to more general systems, its error analysis is structurally simpler than the one of DRKF. Section \ref{sec:fidelity} starts our discussion by first showing the second moment of RKF error is bounded by $C_n^+$ in Theorem \ref{thm:deterministic}, and then the dissipation of the Mahalanobis error through Theorem \ref{thm:dissmaha},  where a more pragmatic online Assumption \ref{aspt:working} formalizes \eqref{eqn:ideal}. A direct Corollary \ref{cor:expstable} shows that the filter is exponentially stable for the mean sequence.   The additional structural complexity of DRKF comes from the fact that the small scale fluctuation sequence  is not an independent one. Section \ref{sec:DRKF} resolves this issue by proving Theorem \ref{thm:errorDRKF}. Section \ref{sec:offline} introduces some intrinsic performance criteria for the reduced filters. This is carried out by a comparison with the  Kalman filters for the inflated systems \eqref{sys:inflatedDRKF} and \eqref{sys:inflated}. The details are in Proposition \ref{prop:DRKF} and Theorem \ref{thm:cutoff}. Immediate corollaries for RMS and accuracy are also drawn there. Section \ref{sec:general} generalizes this idea to more general stochastic settings. Finally, Section \ref{sec:example} reviews some classical methods to control the Kalman filter covariance, and applies it to stochastic turbulence in Fourier domain in various dynamical and observational settings. The related complexity estimates, convergence to stationary Kalman covariance and some matrix inequalities are discussed in the supplementary material.

Before we start the discussion, here are a few standard notations we will use in the following.  $\|C\|$ denotes the spectral norm of a matrix $C$, and $|x|$ is the $l^2$  norm of a vector $x$.  We use $x\otimes x$ to denote the rank $1$ matrix $xx^T$ generated by a column vector $x$. We use $C\in PD (PSD)$ or simply $C$ is PD (PSD) to indicate a symmetric matrix $C$ is positive definite (semidefinite). $[C]_{j,k}$ denotes the $(j,k)$-th coordinate of a matrix $C$, and $[C]_{I^2}$ is the sub-matrix with both indices in a set $I$. And $A\preceq B$ indicates that $B-A\in PSD$. $\lceil a\rceil$ is the smallest integer above a real number $a$. 

We assume the filter initializations are known and of deterministic values. Generally speaking, there are no specific requirements for their values. But some results implicitly rely on the invertibility of the covariance matrices. 

Following \cite{Bou93}, we say a random sequence $Z_0,Z_1,\ldots$ is \emph{stationary}, if $(Z_0,Z_1,\ldots)$ and $(Z_k,Z_{k+1},\ldots)$ have the same distribution. We say such sequence is \emph{ergodic}, if there is only one invariant measure for the shifting map $(Z_0,Z_1,\ldots)\mapsto (Z_1,Z_2,\ldots)$.

There will be three filterations in our discussion. The first one contains all the information of system coefficients up to time $n$, and the initial covariance for the filters:
\[
\mathcal{F}^c_n=\sigma\{A_{k}, B_{k}, \Sigma_{k}, H_{k}, \sigma_{k}, k\leq n\}\vee \sigma\{R_0,  C_0,  C^L_0, V^S_0, \Rtilde_0, \Rtilde^L_0\}.
\]
 Noticeably, all the filter systems have their covariance inside this filteration:
 \[
 \sigma\{R_k, C^L_k, \Rtilde^L_k, C_k, \Rtilde_k, k\leq n+1\}\subset\mathcal{F}^c_n. 
 \]
We will use $\mathcal{F}^c=\vee_{n\geq 0} \mathcal{F}^c_n$ to denote all the information regarding the system coefficients through the entire time line. When the system coefficient and initial filter covariances  are deterministic, $\mathcal{F}^c$ is trivial, so $\E_{\mathcal{F}^c}=\E$.  

The second filteration in addition includes information of the observation and mean initialization
\[
\mathcal{F}^o_n=\sigma\{Y_{k}, k\leq n\}\vee \sigma\{m_0, \mu_0,\mu^L_0\}\vee \mathcal{F}^c_n.
\]
This  filteration also contains the filter mean sequence $m_n, \mu_n, \mu^L_n$. The last filteration contains all the information of system \eqref{sys:random} up to time $n$, $
\mathcal{F}_n=\mathcal{F}^c_n\vee \sigma\{\zeta_k, \xi_k, k\leq n\}.$ 
We use $\E_n Z$, $\E_{\mathcal{F}}Z$ to denote the conditional expectation of a random variable $Z$ with respect to $\mathcal{F}_n$ or another fixed $\sigma$-field $\mathcal{F}$ respectively. 
%
%\begin{rem}
%The fixed subspace projection $\bfP_L$ would require stronger condition for Assumption \ref{aspt:working} to hold, majorly because the cross covariance of the two scales, $\bfP_L \Kalman(\Chat_n)\bfP_S$ is not zero. In order to control this term, we essentially applied the Young's inequality, 
%\[
%(r-1)\bfP_L \Kalman(\Chat_n)\bfP_L+\frac{\bfP_S \Kalman(\Chat_n)\bfP_S}{r-1}
%\succeq \bfP_L\Kalman(\Chat_n)\bfP_S+\bfP_S \Kalman(\Chat_n)\bfP_L. 
%\]
%This would require $\bfP_S \Kalman(\Chat_n)\bfP_S\preceq (r-1)D_S$, which is quite restrictive when $r$ is close to one. 
%\end{rem}
%

\section{Covariance fidelity of RKF}
\label{sec:fidelity}
In the RKF formulation \eqref{sys:RKF}, the multiplicative inflation $r>1$ in large scale, and the constant inflation in small scale $D_s$, intend to remedy the side effect of large scale projection $\bfP_L$ and ensure the covariance estimate does not decrease after the dimension reduction. To be more pragmatic, we measure the actual covariance underestimation caused by this dimension reduction step, through the following sequence of ratios:
\begin{equation}
\label{eqn:beta}
\beta_{n+1}=\sup\{b\geq 0, \Kalman(\Chat_{n+1})\preceq bC^+_{n+1} \}. 
\end{equation}
Intuitively, if this sequence is bounded below from one eventually, $C_{n+1}^+$ does not underestimate the error covariance. More formally, we assume
\begin{aspt}[Acceptable reduction]
\label{aspt:working}
We say the dimension reduction in RKF is asymptotically acceptable if  there is  a finite adjustment time $n_0$ and a $\beta^*<1$ such that 
\[
 \beta_n\leq \beta^* \quad \text{ for all }n\geq n_0. 
\] 
Moreover, we say the RKF enters the acceptable reduction phase, when $n\geq n_0$. 
\end{aspt}
Noticeably, this is an online criterion, so its verification requires an implementation of RKF. Section \ref{sec:offline} will provide an  a priori criterion Assumption \ref{aspt:cutoff} that is sufficient for Assumption \ref{aspt:working}.

\subsection{Second moment of error with system independent noises}
In many scenarios,  the system noises depend on the system coefficients only through $\Sigma_n$ and $\sigma_n$. Precisely speaking:
\begin{equation}
\label{eqn:independentnoise}
\xi_{n+1}\sim\mathcal{N}(0, \Sigma_n),\quad \zeta_{n+1}\sim \mathcal{N}(0,\sigma_n) \quad \text{conditioned on } \mathcal{F}^c\vee\mathcal{F}_n.
\end{equation}
For simplicity, we will describe \eqref{eqn:independentnoise} simply as  \emph{the system noises are independent of the system coefficients}. In the classical setting for Kalman filtering, where the system coefficients are deterministic, this holds automatically. But it may fail in some conditional Gaussian systems. Using \eqref{eqn:independentnoise}, the monotonicity of Kalman updates operator,  the second moment of error $\E_{\mathcal{F}^c} e_n\otimes e_n$ is traceable, and  is in fact bounded by the effective covariance   estimator $C^+_n$. 
\begin{theorem}
\label{thm:deterministic}
Suppose the system noises are independent of the system coefficients, so \eqref{eqn:independentnoise} holds. For any fixed inflation ratio $r>1$, consider applying the RKF \eqref{sys:RKF} to  system \eqref{sys:random}.  Suppose the dimension reduction in RKF is asymptotically acceptable as described in Assumption \ref{aspt:working}.  Then with any fixed initial conditions, when 
\[
n\geq n_0+\lceil -\log  \|[\E_{\mathcal{F}^c} e_{n_0}\otimes e_{n_0}][C^+_{n_0}]^{-1}\|/\log \beta^*\rceil,
\]
the  second moment of the error $e_n=X_n-\mu_n$  is dominated by the covariance estimator:
\[
 \E_{\mathcal{F}^c} e_n\otimes e_n\preceq  C^+_n\quad a.s..
\]
If we take average of both hands, this implies that $\E e_n\otimes e_n\preceq \E C^+_n$. 
\end{theorem}
Notice that the dependence of $n$ on $\E_{\mathcal{F}^c} e_{n_0}\otimes e_{n_0}$ is logarithmic, so in practice the exact value of $\E_{\mathcal{F}^c} e_{n_0}\otimes e_{n_0}$ is not very important. 
\begin{proof}
Define the following sequence using Lemma \ref{lem:norm}
\[
\psi_n=\inf\{\psi: \E_{\mathcal{F}^c}( e_n\otimes e_n)\preceq \psi C_n^+\}=\|\E_{\mathcal{F}^c}( e_n\otimes e_n) [C_n^+]^{-1}\|,\quad n\geq m.
\]
We claim that 
\begin{equation}
\label{tmp:psi}
\psi_{n+1}\leq  \max\{1,\psi_n\beta_{n+1}\}.
\end{equation}
Then by Assumption \ref{aspt:working}, for $n\geq n_0$, $\psi_{n+1}$ converges to $1$ geometrically with ratio $\beta^*$, so the claim of this theorem holds. 

In order to show \eqref{tmp:psi}, consider the forecast error $\ehat_{n+1}=X_{n+1}-(A_n \mu_n+B_n)$. The following recursion can be established:
\[
\ehat_{n+1}=A_{n}e_n +\xi_{n+1},\quad e_{n+1}=(I-\Khat_{n+1} H_{n}) \ehat_{n+1}-\Khat_{n+1}\zeta_{n+1}.  
\]
In combination:
\[
e_{n+1}=(I-\Khat_{n+1}H_{n}) A_{n}e_{n} +(I-\Khat_{n+1}H_n) \xi_{n+1}-\Khat_{n+1}\zeta_{n+1}. 
\]
Because $(I-\Khat_{n+1}H_{n})A_n\in \mathcal{F}^c$, where  $\xi_{n+1}$ and $\zeta_{n+1}$ are conditionally mean zero based on \eqref{eqn:independentnoise}, we find 
\begin{align}
\notag
\E_{\mathcal{F}^c} e_{n+1}\otimes &e_{n+1}=\E_{\mathcal{F}^c}[(I-\Khat_{n+1}H_{n})(A_n(e_n\otimes e_n)A_n^T+\Sigma_n) (I-\Khat_{n+1}H_{n})^T 
 + \Khat_{n+1}^T\sigma_n\Khat_{n+1}]\\
 \label{tmp:errorbound}
 &=[(I-\Khat_{n+1}H_{n})(A_n\E_{\mathcal{F}^c} e_n\otimes e_n A_n^T+\Sigma_n) (I-\Khat_{n+1}H_{n})^T 
 + \Khat_{n+1}^T\sigma_n\Khat_{n+1}]\\
 \notag
&\preceq (I-\Khat_{n+1}H_{n})(\psi_nA_nC^+_nA_n^T+\Sigma_n) (I-\Khat_{n+1}H_{n})^T 
 + \Khat_{n+1}^T\sigma_n\Khat_{n+1}\\
 &\preceq 
 \notag
\psi_n [(I-\Khat_{n+1}H_{n})(A_nC_n^+A_n^T+\Sigma_n) (I-\Khat_{n+1}H_{n})^T 
 + \Khat_{n+1}^T\sigma_n \Khat_{n+1}]=\psi_n\Kalman ( \Chat_{n+1} ).
\end{align}
In the penultimate step, we used that $\psi_n\geq 1$, and also the well known matrix identity for Kalman update
\[
\Kalman(\Chat_{n+1})=(I-\Khat_{n+1}H_{n})\Chat_{n+1}(I-\Khat_{n+1}H_{n})^T 
 + \Khat_{n+1}^T\sigma_n \Khat_{n+1}.
\]
 By the definition of $\beta_{n+1}$ \eqref{eqn:beta}, we have $
\psi_n\Kalman ( \Chat_{n+1} )\preceq \psi_n\beta_{n+1} C_{n+1}\preceq \psi_{n+1}C_{n+1}.$ 
\end{proof}
\begin{rem}
In fact, if $\E_{\mathcal{F}^c} e_0=0$, one can also show $\E_{\mathcal{F}^c} e_n=0$ in this setting, so $\E_{\mathcal{F}^c} e_n\otimes e_n$ is actually the error covariance. But rigorously speaking, the RKF mean $\mu_n$ is a biased estimator. Unbiasedness  would  require $\E_{\mathcal{F}^o_n} e_n=0$ a.s., which in general does not hold. In order to avoid confusions, we did not mention this fact in the theorem.  
\end{rem}

\subsection{Mahalanobis error dissipation}
\label{sec:dissmaha}
If the system coefficients have dependence on the system noises,  the Kalman gain matrix $\Khat_{n+1}$ may have correlation with the error term $e_n$. So the identity \eqref{tmp:errorbound} no longer holds, and the second moment of the error is not traceable.  But even in this difficult scenario,  an intrinsic matrix inequality still holds. In the context of the optimal Kalman filter \eqref{sys:optimal}, it can be formulated as 
\[
A_n^T(I-K_{n+1}H_n)^TR_{n+1}^{-1}(I-K_{n+1}H_n)A_n\preceq R_n^{-1},
\] 
and for RKF it becomes \eqref{tmp:pin} in below. From this perspective, the Mahalanobis error  $\|e_n\|^2_{C_n^+}$ is a natural statistics that dissipates through time.

\begin{theorem}
\label{thm:dissmaha}
For any fixed inflation $r>1$, consider applying the RKF \eqref{sys:RKF} to  system \eqref{sys:random}.  Suppose the dimension reduction in RKF is asymptotically acceptable as described by Assumption \ref{aspt:working}, then 
\[
\E \|e_n\|_{C_n^+}^2\leq (\beta^*)^{n-n_0}\E \|e_{n_0}\|^2_{C^+_{n_0}}+\frac{2d}{1-\beta^*}. 
\]
In other words, the Mahalanobis error is dissipative after the transition time $n_0$. 
\end{theorem}
\begin{proof}
We will show that given any $n$,
\begin{equation}
\label{eqn:diss}
\E_n \|e_{n+1}\|^2_{C_{n+1}^+}\leq \beta_{n+1} \|e_n\|^2_{C_n^+}+ 2d\beta_{n+1}.
\end{equation}
Then the original claim of this theorem can be achieved by applying the Gronwall's inequality in discrete time. To show \eqref{eqn:diss}, recall that in the proof of Theorem \ref{thm:deterministic}, the filter error has the following recursion:
\[
e_{n+1}=(I-\Khat_{n+1} H_{n})A_ne_n+(I-\Khat_{n+1} H_{n}) \xi_n-\Khat_{n+1}\zeta_{n+1}.
\]
Since $\xi_{n+1}$ and $\zeta_{n+1}$ are independent of $\mathcal{F}_n$ conditioned on $\Sigma_n$ and $\sigma_n$, we find that 
\begin{align}
\label{tmp:en}
\E_n e_{n+1}^T&[C^+_{n+1}]^{-1} e_{n+1}=\E_n e_n^TA^T_n(I-\Khat_{n+1}H_{n})^T [C^+_{n+1}]^{-1} (I-\Khat_{n+1}H_{n}) A_{n}e_n\\
\label{tmp:noise}
&+\E_n\xi_{n+1}^T (I-\Khat_{n+1}H_n)^T[C^+_{n+1}]^{-1} (I-\Khat_{n+1}H_{n}) \xi_{n+1}+\E_n \zeta^T_{n+1}\Khat_{n+1}^T[C^+_{n+1}]^{-1}\Khat_{n+1}\zeta_{n+1}. 
\end{align}
For the first part \eqref{tmp:en}, we claim that   
\begin{equation}
\label{tmp:pin}
A_{n}^T(I-\Khat_{n+1}H_{n})^T [C^+_{n+1}]^{-1}(I-\Khat_{n+1}H_{n}) A_{n}\preceq \beta_{n+1} [C_n^+]^{-1} .
\end{equation}
To see that, notice by \eqref{eqn:beta}  $\beta_{n+1}C^+_{n+1}\succeq \Kalman(\Chat_{n+1}) \succeq(I-\Khat_{n+1}H_{n})\Chat_{n+1} (I-\Khat_{n+1}H_n)^T.$
Moreover $(I-\Khat_{n+1}H_{n})=(I+\Chat_{n+1}H_n^T\sigma^{-1}_nH_n)^{-1}$ is clearly invertible. The inversion of the inequality above reads
\begin{equation}
\label{tmp:Kalcor}
(I-\Khat_{n+1}H_{n})^T[C^+_{n+1}]^{-1}(I-\Khat_{n+1}H_{n})\preceq\beta_{n+1} \Chat_{n+1}^{-1}.
\end{equation}
Next, notice that $\Chat_{n+1}\succeq A_n C_n^+ A_n^T$, so 
\[
 A_n(I-\Khat_{n+1}H_{n})[C_{n+1}^+]^{-1} (I-\Khat_{n+1}H_{n})A_n\preceq \beta_{n+1} \Chat_{n+1}^{-1}\preceq \beta_{n+1} A_n [A_n C^+_n A_n^T]^{-1}A_n^T,
\]
which by  Lemma \ref{lem:matrix} leads to \eqref{tmp:pin}. To deal with \eqref{tmp:noise}, we use the identity $a^T A a=\text{tr}(A a a^T)$ and the independence of $\xi_{n+1},\zeta_{n+1}$,
\begin{align*}
&\E_n \xi_{n+1}^T (I-\Khat_{n+1}H_n)^T[C^+_{n+1}]^{-1} (I-\Khat_{n+1}H_{n}) \xi_{n+1}+\zeta^T_{n+1}\Khat_{n+1}^T[C^+_{n+1}]^{-1}\Khat_{n+1}\zeta_{n+1}\\
&=\E_n\text{tr} [(I-\Khat_{n+1}H_{n})[C^+_{n+1}]^{-1} (I-\Khat_{n+1}H_n)^T\Sigma_n+ \Khat_{n+1}\sigma_n\Khat_{n+1}^T[C^+_{n+1}]^{-1}].
\end{align*}
Note that by definition, $\Chat_{n+1}\succeq \Sigma_n$, so Lemma \ref{lem:matrix} implies:
  \[
  \text{tr} [(I-\Khat_{n+1}H_{n})[C^+_{n+1}]^{-1} (I-\Khat_{n+1}H_n)^T\Sigma_n]\leq d\beta_{n+1}.
  \] 
Also notice that
\begin{equation}
\label{tmp:Kalsigma}
\Kalman(\Chat_{n+1})=(I-\Khat_{n+1}H_n)\Chat_{n+1}(I-\Khat_{n+1}H_n)^T+\Khat_{n+1}\sigma_n\Khat_{n+1}^T\succeq\Khat_{n+1}\sigma_n\Khat_{n+1}^T. 
\end{equation}
Then by $\beta_{n+1}C^+_{n+1}\succeq  \Kalman(\Chat_{n+1})\succeq \Khat_{n+1}\sigma_n\Khat_{n+1}^T$, 
$\text{tr}(\Khat_{n+1}\sigma_n\Khat_{n+1}^T[C^+_{n+1}]^{-1})\leq d\beta_{n+1}.$
By summing up \eqref{tmp:en} and \eqref{tmp:noise}, we have reached \eqref{eqn:diss} and so ends the proof. 
\end{proof}
\begin{rem}
\label{rem:inflationRKF}
In the analysis of the standard Kalman filter and extended Kalman filter, \cite{Bou93, RGYU99} implicitly exploited the same mechanism but does not require a multiplicative inflation. The price they paid is that they require the covariance sequences  $C_n, C_n^{-1},\Sigma_n$ and $\Sigma_n^{-1}$ to be bounded both from above. (\cite{Bou93} has weaker assumptions, but its results are qualitative rather than quantitative). With some extra works, we can as well removes the multiplicative inflation by adding similar conditions. But such conditions are usually very bad in  high dimensional settings, as $\Sigma_n$ may have many small scale entries being very close to zero. 
\end{rem}
\subsection{Exponential stability}
Another useful property implied by the previous analysis is that RKF is exponentially stable. Let $(\mu_0, C_0)$ and $(\mu'_0, C_0)$ be two implementations of RKF with the same covariance but different means. Then these two RKFs share the same covariance estimate, and the difference in their mean estimates is given by
\[
(\mu_n-\mu_n')= U_{n,0}(\mu_0-\mu_0'),\quad U_{n,m}=\prod_{k=m}^{n-1} (I-\Khat_{k+1}H_k)A_k. 
\]
So if $\|U_{n,0}\|$ converges to zero exponentially fast, then so does the mean difference. In \cite{Bou93}, this is called the exponential stability. 

\begin{cor}
\label{cor:expstable}
Under the conditions of  Theorem \ref{thm:dissmaha}, suppose also that 
$\sup_{n}\|C_{n}^+\|<\infty, \|[C_{0}^+]^{-1}\|<\infty,$ 
then the RKF filter is exponentially stable as 
\[
\limsup_{n\to \infty} \frac{1}{n} \log \left\|\prod_{k=0}^{n-1} (I-\Khat_{k+1}H_k)A_k\right\|\leq \frac{1}{2}\log \beta^*.
\]
\end{cor}
\begin{proof}
Let $U_{n,n_0}=\prod_{k=n_0}^{n-1}(I-\Khat_{k+1}H_k)A_k$. By iterating \eqref{tmp:pin} $n$ times, we find that
\[
\|C_n^+\|^{-1}U^T_{n,0}U_{n,0}\preceq U_{n,0}^T[C_n^+]^{-1}U_{n,0}\preceq\left( \prod_{k=1}^n\beta_n\right)[C_{0}^+]^{-1}. 
\]
Taking spectral norm on both hand side yields our claim. 
\end{proof}
Sections \ref{sec:offline} and \ref{sec:example} will discuss how to bound $\|C_n^+\|$. 

\section{Covariance fidelity of DRKF}
\label{sec:DRKF}
In the dynamical decoupled scenario \eqref{eqn:blockdiag}, DRKF has a significant advantage  comparing with RKF: since no large scale projection is applied, there is no risk of underestimating the error covariance, so online criteria like Assumption \ref{aspt:working} are not necessary. The disadvantages are two folds, first it has a special dynamical structural requirement, second the small scale fluctuation requires more technical treatments. To see the second point, it is straight forward to have the following recursion for the filter error, just like in the proof of Theorem \ref{thm:deterministic},
\[
e^L_{n+1}=(I-K^L_{n+1}H^L_{n}) A_{n}e^L_{n} +(I-K^L_{n+1}H^L_n) \xi^L_{n+1}- K^L_{n+1}\zeta_{n+1}-K^L_{n+1}H^S_n\Delta X^S_{n+1}. 
\]
Unlike $\xi_{n+1}$ and $\zeta_{n+1}$, in most situations, $\Delta X^S_{n+1}=X^S_{n+1}-\mu^S_{n+1}$ has a nonzero correlation with the error $e^L_n$, as it is not an independent time series.  Therefore the second moment matrix is not traceable because  \eqref{tmp:errorbound} no longer holds. On the other hand, the Mahalanobis error dissipation holds as a much more stable mechanism. In order to show that, we need additional conditions on the small scale dynamics $A_n^S$, and impose \eqref{eqn:independentnoise} type of independence condition on the small scale system. Fortunately, these conditions  hold for many important examples in Section \ref{sec:example}, and are trivial for a deterministic constant stable asymptotic covariances for the small scales. 
\begin{theorem}
\label{thm:errorDRKF}
Consider applying DRKF \eqref{sys:DRKF} to system \eqref{sys:random} with two-scale dynamical decoupling \eqref{eqn:blockdiag}. Suppose there is a spectral gap  $\lambda_S<1$ such that
\[
A^S_{k,j} V^S_j  (A^S_{k,j})^T\preceq \lambda_S^{k-j} V^S_k, \quad A^S_{k,j}=A^S_{k-1}\cdots A^S_{j+1}A^S_j. 
\]
Assume also the distribution of the small scale system noise $\xi^S_n$ is $\mathcal{N}(0,\sigma_n^S)$, conditioned on the system coefficients $\sigma$-field $\mathcal{F}^c$.  Then the following holds
\begin{equation}
\label{eqn:errorDRKF}
\E\|e^L_n\|^2_{C^L_n}\leq \frac{2}{r^n}\E\|e_0^L\|^2_{C^L_0}+\frac{2p(1+\gamma_\sigma)}{r-1}+\frac{4\sqrt{\lambda_S r}p\gamma_\sigma}{(\sqrt{r}-1)(1-\sqrt{\lambda_S})}
\end{equation}
The last term comes from the time correlated small scale fluctuation, and the constant $\gamma_\sigma$  is given by
\[
\gamma_\sigma=\sup_{n\geq 0 }\{\|[\sigma_n^L]^{-1}H^S_nV^S_{n+1} (H^S_n)^T\|\}.
\] 
Note that $\gamma_\sigma\leq 1$, and it has the potential to be small if $H^S_nV^S_n (H^S_n)^T$ is small. 
\end{theorem}
\begin{proof}
The filter error follows the recursion:
\[
e^L_{n+1}= (I-K^L_{n+1}H_n^L)A_n^L e^L_n+ K^L_{n+1}\zeta_{n+1}-K^L_{n+1}H_n^L \xi_{n+1}-K^L_{n+1}H_n^S\Delta X^S_{n+1}.
\]
In order to take away the influence of $\Delta X_n^S$, consider 
\[
\etilde^L_n=e_n^L-\sum_{k=1}^n  U^L_{n,k}Q^S_k,
\quad 
Q^S_k:=K^L_{k}H_{k-1}^S\Delta X_{k}^S,\quad U^L_{n,k}:=(I-K^L_{n} H_{n-1})A_{n-1}^L\cdots(I-K^L_{k+1} H_{k}^L)A_k^L.
\]
$\etilde^L_n$ follows the recursion
\[
\etilde^L_{n+1}= (I-K^L_{n+1}H_n^L)A_n^L \etilde^L_n+ K^L_{n+1}\zeta_{n+1}-K^L_{n+1}H_n^L \xi_{n+1}. 
\]
Then the proof of Theorem \ref{thm:dissmaha} is valid for $\etilde^L_{n+1}$ completely the same, as long as we replace $\beta_n$ with $\frac{1}{r}$. In place of \eqref{eqn:diss}, we have
\[
\E_n \|\etilde^L_{n+1}\|_{C_{n+1}^L}^2\leq \frac{1}{r} \|\etilde^L_n\|_{C_n^L}^2+\frac{2d}{r}. 
\]
As a consequence of the Gronwall's inequality, $\E \|\etilde^L_{n}\|^2_{C_n^L}\leq \frac{1}{r^n}\|\etilde^L_0\|^2_{C_0^L}+\frac{2d}{r-1}. $ Because of Young's inequality 
\[
\E \|e^L_n\|^2_{C_n^L}\leq 2\E \|\etilde^L_n\|^2_{C_n^L}+2\E\left\| \sum_{k=1}^n  U^L_{n,k} Q_{k}^S\right\|^2_{C_n^L}.
\]
It suffices for us to bound 
\begin{equation}
\label{tmp:sum}
\E \left\| \sum_{k=1}^n  U^L_{n,k} Q_{k}^S\right\|^2_{C_n^L}
= \sum_{j,k\leq n} \E (Q_j^S)^T(U_{n,j}^L)^T [C^L_n]^{-1}U_{n,k}^L Q_{k}^S.
\end{equation}
Following the proof of  \eqref{tmp:pin}, a similar matrix inequality also holds  for ERKF, 
\[
(A^L_n)^T(I-K_{n+1}^L H_n^L)^T[C_{n+1}^L]^{-1}(I-K_{n+1}^L H_n^L)A^L_n\preceq \frac{1}{r}[C_n^L]^{-1}.
\]
Therefore we have $(U^L_{n,k})^T[C_n^L]^{-1}U^L_{n,k}\preceq r^{k-n} [C_k^L]^{-1}$. 

\par The terms in the sum \eqref{tmp:sum} with $j=k$ can be bounded by
\begin{align*}
\E (Q_k^S)^T(U_{n,k}^L)^T [C^L_n]^{-1}U_{n,k}^L Q_{k}^S
&\leq \frac{1}{r^{n-k}}\E (Q_k^S)^T[C^L_k]^{-1}Q_k^S\\
&=\frac{1}{r^{n-k}}\E \text{tr}((K_k^L)^T [C^L_k]^{-1}K_k^L(H_{k-1}^S\Delta X^S_k\otimes H_{k-1}^S\Delta X^S_k))\\
&\leq \frac{1}{r^{n-k}}\E \text{tr}((K_k^L)^T [C^L_k]^{-1}K_k^L\E_{\mathcal{F}^c_k}(H_{k-1}^S\Delta X^S_k\otimes H_{k-1}^S\Delta X^S_k))\\
&= \frac{1}{r^{n-k}}\E \text{tr}( [C^L_k]^{-1}K_k^L H_{k-1}^SV^S_k (H_{k-1}^S)^T(K_k^L)^T).
\end{align*}
Similar to \eqref{tmp:Kalsigma},  we have 
\begin{equation}
\label{tmp:sigma2}
C^L_k=r\Kalman_L(\Chat^L_k)\succeq rK_k^L\sigma^L_{k-1} (K_k^L)^T\succeq \gamma_\sigma^{-1}K_k^L H_{k-1}^SV^S_k (H_{k-1}^S)^T(K_k^L)^T.
\end{equation}
As a consequence of Lemmas \ref{lem:trace} and \ref{lem:trivial}, the $j=k$ terms in \eqref{tmp:sum} can be further bounded by 
\begin{equation}
\label{tmp:1}
\E (Q_k^S)^T(U_{n,k}^L)^T [C^L_n]^{-1}U_{n,k}^L Q_{k}^S\leq \frac{\gamma_\sigma}{r^{n-k}} \E \text{tr}(I_p)=\frac{p\gamma_\sigma}{r^{n-k}}. 
\end{equation}
The $j<k$ terms in \eqref{tmp:sum} come from time correlations of $\Delta X^S_k$. In order to bound them, notice that:
\begin{align*}
\E (Q_j^S)^T(U_{n,j}^L)^T [C^L_n]^{-1}U_{n,k}^L Q_{k}^S
&=\E \text{tr}((H^S_{j-1})^T(K_{j}^L)^T(U_{n,j}^L)^T[C^L_n]^{-1}U^L_{n,k} K_{k}^LH^S_{k-1})(\Delta X^S_k\otimes \Delta X^S_j))\\
&=\E \text{tr}(W_j^T(U_{n,j}^L)^T[C^L_n]^{-1}U^L_{n,k} W_k)\E_{F^c_n}(\Delta X^S_k\otimes \Delta X^S_j))
\end{align*}
with $W_k:=K_k^L H^S_{k-1}$. The  moving average representation of $\Delta X_k^S$ is:
\[
\Delta X_k^S=A^S_{k,j}\Delta X_j^S+\sum_{i=j+1}^{k}A^S_{k,i}\xi^S_i.
\]
Since for $i>j$, $\xi^S_k$ is distributed as $\mathcal{N}(0, \sigma_i^S)$ conditioned on  $\mathcal{F}^c_n\vee \mathcal{F}_j$, 
\[
\E_{\mathcal{F}^c_n\vee \mathcal{F}_j}(\Delta X^S_k\otimes \Delta X^S_j)=
\E_{\mathcal{F}^c_n\vee \mathcal{F}_j}(A^S_{k,j}\Delta X^S_j\otimes \Delta X_j^S )= A_{k,j}^SV^S_j.
\]
Therefore
\begin{align}
\notag
&\E (Q_j^S)^T(U_{n,j}^L)^T [C^L_n]^{-1}U_{n,k}^L Q_{k}^S+(Q_k^S)^T(U_{n,k}^L)^T [C^L_n]^{-1}U_{n,j}^L Q_{j}^S\\
\label{tmp:WCW}
&=\E \text{tr}([W_j^T(U_{n,j}^L)^T[C^L_n]^{-1}U^L_{n,k} W_kA_{k,j}^S+
(A_{k,j}^S)^TW_j^T(U_{n,j}^L)^T[C^L_n]^{-1}U^L_{n,k} W_k]V^S_j).
\end{align}
In order to apply Lemma \ref{lem:trace}, we are interested in bounding the symmetric matrix 
\[
Z_{j,k}:=W_j^T(U_{n,j}^L)^T[C^L_n]^{-1}U^L_{n,k} W_kA_{k,j}^S+
(A_{k,j}^S)^TW_k^T(U_{n,k}^L)^T[C^L_n]^{-1}U^L_{n,j} W_j.
\]
Notice that for any PSD matrix $C$, matrices $A$ and $B$, and $\gamma>0$, the following holds
\[
(\gamma^{-1} A-\gamma B)^T C(\gamma^{-1}A-\gamma B)\succeq 0\quad\Rightarrow\quad 
\gamma^{-2}A^TCA+\gamma^2 B^TCB\succeq A^TC B+B^T CA.
\]
For our purpose, let $A=U^L_{n,k} W_kA_{k,j}^S, B=U^L_{n,j} W_j, C=[C^L_n]^{-1}, \gamma=(\lambda_Sr)^{\frac{k-j}{4}},$
and find 
\[
Z_{j,k}\preceq \gamma^{-2}(A_{k,j}^S)^T W_k^T(U^L_{n,k})^T [C^L_n]^{-1}U^L_{n,k} W_kA_{k,j}^S
+\gamma^{2} W_j^T(U^L_{n,j})^T [C^L_n]^{-1}U^L_{n,j} W_j.
\]
To continue, recall that $(U^L_{n,k})^T [C^L_n]^{-1}U^L_{n,k}\preceq r^{k-n} [C^L_k]^{-1}$, and the relation \eqref{tmp:sigma2}. Then using Lemmas \ref{lem:trace} and \ref{lem:trivial},
\begin{align*}
\text{tr}((A_{k,j}^S)^T W_k^T(U^L_{n,k})^T [C^L_n]^{-1}U^L_{n,k} W_k A_{k,j}^SV^S_j)
&\leq r^{k-n}\text{tr}((A_{k,j}^S)^T(K_k^LH^S_{k-1})^T[C^L_k]^{-1}K_k^LH^S_{k-1}A_{k,j}^SV^S_j)\\
&= r^{k-n}\text{tr}([C^L_k]^{-1}K_k^LH^S_{k-1}A_{k,j}^SV^S_j (A_{k,j}^S)^T(K_k^LH^S_{k-1})^T)\\
&\leq \lambda_S^{k-j}r^{k-n}\text{tr}([C^L_k]^{-1}K_k^LH^S_{k-1}V^S_k(K_k^LH^S_{k-1})^T).
\end{align*}
Using \eqref{tmp:sigma2} again, we find the quantity above is bounded by $\gamma_\sigma \lambda_S^{k-j}r^{k-n}p$.  Likewise 
\[
 \text{tr}(W_j^T(U^L_{n,j})^T [C^L_n]^{-1}U^L_{n,j} W_jV^S_j)\leq r^{j-n}\text{tr}([C^L_j]^{-1}K_j^LH^S_{j-1}V^S_j(K_j^LH^S_{j-1})^T)\leq r^{j-n}\gamma_\sigma p. 
 \] As a consequence,
\begin{equation}
\label{tmp:2}
\eqref{tmp:WCW}
\leq \E \text{tr}(Z_{j,k} V^S_j)\leq 2r^{\frac{k+j}{2}-n}\lambda_S^{\frac{k-j}{2}} \gamma_\sigma p. 
\end{equation}
Finally, we can bound \eqref{tmp:sum} by \eqref{tmp:1} and \eqref{tmp:2}:
\begin{align*}
\eqref{tmp:sum}&\leq \sum_{k=1}^n  \frac{p}{r^{n-k}}+2p\sum_{j<n}r^{\frac{j-n}{2}}\sum_{k\geq j+1} \lambda^{\frac{k-j}{2}}_S
\leq \frac{rp \gamma_\sigma}{r-1}+\frac{2\sqrt{\lambda_S r} p \gamma_\sigma}{(\sqrt{r}-1)(1-\sqrt{\lambda_S})}.
\end{align*}
\end{proof}
\begin{rem}
In the first appearance, the formulation of the result may suggest the bigger the inflation strength $r$, the smaller the filter error. In fact, this is an artifact caused by the usage of Mahalanobis error. The covariance estimator $C^L_n$ may have a super linear growth with respect to $r$. This is slightly discussed in the next section. On the other hand, bigger $r$ does imply stronger stability. 
\end{rem}
Following the same proof of Corollary \ref{cor:expstable}, the exponential stability holds for DRKF as well:
\begin{cor}
\label{cor:expstableD}
Under the conditions of  Theorem \ref{thm:errorDRKF}, suppose also that 
\[
\sup_{n}\|C_{n}^L\|<\infty,\quad \|[C_{0}^L]^{-1}\|<\infty, 
\] 
then the DRKF filter is exponentially stable as 
\[
\limsup_{n\to \infty} \frac{1}{n} \log \left\|\prod_{k=0}^{n-1} (I-K^L_{k+1}H^L_k)A^L_k\right\|\leq -\frac{1}{2}\log r.
\]
\end{cor}

\section{Intrinsic performance criteria}
\label{sec:offline}
Sections \ref{sec:fidelity} and \ref{sec:DRKF} have demonstrated the covariance fidelity of the reduced filters. But in order for these results to be applicable to concrete  problems, there are two issues:
\begin{itemize}
\item The Mahalanobis error is informative for the filter error only when the  covariance estimators $C^L_n$ and $C^+_n$ are bounded, and so is the claim that $\E_{\mathcal{F}^c }e_n\otimes e_n\preceq C^+_n$ in Theorem \ref{thm:deterministic}. 
\item Since RKF takes a large scale projection, it requires Assumption \ref{aspt:working}. This is an online criterion that can be verified only by implementing RKF. In practice, a priori criteria are more useful for verifications. 
\end{itemize}
Preferably, both questions should be answered independent of the reduced filters initialization. In this way, we capture the intrinsic reduced filter performance for system \eqref{sys:random}.

The idea here is quite simple, we will consider two signal-observation systems \eqref{sys:inflatedDRKF} and \eqref{sys:inflated} as augmentations of system \eqref{sys:random}, and use their Kalman filters covariance $\Rtilde^L_n$ and $\Rtilde_n$ as a performance reference. These reference filters are important for our reduced filters, because Proposition \ref{prop:DRKF} and Theorem \ref{thm:cutoff} show  direct connections between $\Rtilde^L_n, \Rtilde_n$ and $C^L_n, C^+_n$ respectively.  This approach has four advantages:
\begin{enumerate}[1)]
\item  The Kalman filter covariance $R_n$ describes the smallest possible filter covariance for signal-observation systems like \eqref{sys:random}. Since the augmented systems \eqref{sys:inflatedDRKF} and \eqref{sys:inflated} are roughly small perturbations of system \eqref{sys:random}, the associated Kalman filter covariance $\Rtilde^L_n, \Rtilde_n$ are not too different from the proper part of $R_n$ (Proposition \ref{prop:RRtilde} explores some sufficient conditions.) So if $C^L_n$ or $C^+_n$ are bounded by $\Rtilde^L_n, \Rtilde_n$, the reduced filter covariance is comparable with the optimal. 
\item Kalman filters are direct and intrinsic descriptions of how well systems like \eqref{sys:random} can be filtered. The dependence of $\Rtilde^L_n$ or $\Rtilde_n$ on the system coefficients is very nonlinear. Imposing conditions on $\Rtilde_n$ instead of on the system coefficients makes our exposition much simpler.
\item Unlike reduced filters, Kalman filter has been a classical research object for decades. There is a huge literature we can exploit.
\item In particular, Kalman filter covariance converges to a unique stationary solution of the associated Riccati equation in \eqref{sys:optimal}, assuming the system coefficients are ergodic stationary sequences, and other weak conditions hold. See \cite{Bou93} and Section \ref{sec:stationary} for details. This unique solution is independent of the initial condition. By imposing conditions on the stationary solution, our results for the reduced filters are independent of the initial conditions as well. When we refer to this stationary solution in the following discussion, we implicitly assume the existence of this unique stationary solution. 
\end{enumerate}

\subsection{Kalman filters for comparison}
The connection of DRKF with the augmented system \eqref{sys:inflatedDRKF} is simple and direct:
\begin{proposition}
\label{prop:DRKF}
Consider applying DRKF \eqref{sys:DRKF} to system \eqref{sys:random} with dynamical decoupling \eqref{eqn:blockdiag}. Let $\Rtilde^L_n$ be the Kalman filter covariance for the large scale reference system \eqref{sys:inflatedDRKF} with the same system coefficient realization as  in \eqref{sys:random}. If $\Rtilde^L_n=\frac{1}{r} C^L_n$ holds for $n=0$, then it holds for all $n\geq 0$.
\end{proposition}
\begin{proof}
Suppose our claim holds at time $n$. Then the Kalman filter covariance for system \eqref{sys:inflatedDRKF} follows:
\[
\widehat{\Rtilde}^{L}_{n+1}=rA^L_n \Rtilde^L_n (A^L_n)^T +\Sigma_n=A^L_n C^L_n A^L_n +\Sigma_n=\Chat^L_n,
\]
therefore our claim holds at time $n+1$ as well:
\[
\Rtilde^L_{n+1}=\Kalman_L(\widehat{\Rtilde}^{L}_{n+1})=\Kalman_L(\Chat^L_{n+1})=\frac{1}{r}C^L_{n+1}.
\]
\end{proof}
In the case of RKF, we need to consider system \eqref{sys:inflated}, of which the Kalman filter covariance $\Rtilde_n$ follows the recursion
\begin{equation}
\label{sys:controller}
\begin{gathered}
\Rtilde_{n+1}=\Kalman(\widehat{\Rtilde}_{n+1}),\quad \widehat{\Rtilde}_{n+1}=r' A_n \Rtilde_nA_n^T+\Sigma'_n.
\end{gathered}
\end{equation}
The Kalman update operator $\Kalman$ is the same as in \eqref{sys:optimal}. Unlike Proposition \ref{prop:DRKF}, where we showed $\Rtilde^L_n$ is directly a multiple of $C^L_n$, this time $r\Rtilde_n$ will be an upper bound for $C_n$, which leads to $r\Rtilde^{+}_n\succeq C_n^+$.  In addition, using this inequality, we can transfer the online Assumption \ref{aspt:working} to an assumption regarding $\Rtilde_n$:
\begin{aspt}[Reference projection]
\label{aspt:cutoff}
Let $\Rtilde_n$ be a (the unique stationary) positive definite (PD) solution of \eqref{sys:controller}. Assume its small scale part is bounded as below with a $\frac{1}{r}<\beta^*<1$
\[
\bfP_S\Rtilde_n\bfP_S \preceq (\beta^* r-1)D_S.  
\]
\end{aspt}
As we discussed earlier in this section and with more details in Section \ref{sec:stationary}, the Riccati equation \eqref{sys:controller} has a unique stationary solution under weak conditions. Despite that Theorem \ref{thm:cutoff} below works for any solution of \eqref{sys:controller}, by considering the stationary solution it allows Assumption \ref{aspt:cutoff} to be independent of the initial conditions.  

\begin{theorem}
\label{thm:cutoff}
Consider applying RKF \eqref{sys:RKF} to system \eqref{sys:random}, and the Kalman filter covariance $\Rtilde_n$ for system \eqref{sys:inflated}. Then after a finite time $n_0$, the RKF error covariance estimate $C_n^+$ is bounded by the reference covariance $\Rtilde_n$
\[
C_{n}^+\preceq r \Rtilde_n+D_S,\quad n\geq n_0=\lceil \log (\|\Rtilde_0^{-1}C_0\|)/\log(r'/r)\rceil.
\]
If in addition the reference projection Assumption \ref{aspt:cutoff} holds, then Assumption \ref{aspt:working} also holds, and the acceptable reduction phase starts no later than $n_0$.
\end{theorem}
\begin{proof}
 For all $n\geq 0$, denote
\[
\nu_n=\inf\{\nu: C_n\preceq \nu \Rtilde_n\}=\|[\Rtilde_n]^{-1} C_n\|.
\]
Following the formulation of $n_0$, we assume $\nu_0<\infty$. We claim that $\nu_n$ has the following recursive relation:
\begin{equation}
\label{tmp:nu}
\nu_{n+1}\leq r \max\{1, \tfrac{1}{r'}\nu_n\}. 
\end{equation}
This comes from a simple induction. Suppose that $C_n\preceq \nu_n \Rtilde_n$, we have 
\[
\Chat_{n+1}=A_nC_nA_n^T+\Sigma_n+A_n D_S A_n^T
\preceq \nu_n [A_n \Rtilde_nA_n^T]+\Sigma'_n\preceq \tfrac{1}{r'}\nu_n \widehat{\Rtilde}_{n+1}.
\]
Hence by the monotonicity and concavity of the operator $\Kalman$, Lemma \ref{lem:Kalconcave},
\begin{equation}
\label{tmp:postbd}
C_{n+1}\preceq r \Kalman(\Chat_{n+1})\preceq r\Kalman(\tfrac{1}{r'}\nu_n \widehat{\Rtilde}_{n+1}) 
\preceq r\max\{1, \tfrac{1}{r'}\nu_n\} \Kalman( \widehat{\Rtilde}_{n+1})=\nu_{n+1}\Rtilde_{n+1},
\end{equation}
which completes the induction. Then if we iterate \eqref{tmp:nu} $n_0$ time, we find that 
\[
\nu_n \leq r,\quad C_n\preceq r \Rtilde_n,\quad \text{ for all }n\geq n_0.
\]

Next, we prove the second claim of this theorem by showing a stronger result, that is the $\beta_n$ sequence defined  by \eqref{eqn:beta} can be bounded by
\begin{equation}
\label{tmp:betachoice}
\beta_n\leq \frac{\nu_n}{r^2}(\beta^*r-1)+\frac{1}{r}. 
\end{equation}
Then because $\nu_n\leq r$ when $n\geq n_0$, Assumption \ref{aspt:working} is implied.  To see \eqref{tmp:betachoice}, denote 
\[
\bfP_L\Kalman(\Chat_n)\bfP_L=K_{LL},\quad 
\bfP_L\Kalman(\Chat_n)\bfP_S=K_{LS},\quad
\bfP_S\Kalman(\Chat_n)\bfP_L=K_{SL},\quad
\bfP_S\Kalman(\Chat_n)\bfP_S=K_{SS}. 
\]
For any $\beta_n\geq \tfrac{1}{r}$, from
\[
[\sqrt{\beta_n r-1} \bfP_L-\tfrac{1}{\sqrt{\beta_nr-1}}\bfP_S]
\Kalman(\Chat_n)[\sqrt{\beta_n r-1} \bfP_L-\tfrac{1}{\sqrt{\beta_nr-1}}\bfP_S]\succeq 0,
\]
we have
\begin{equation}
\label{eqn:Cauchy}
(\beta_n r-1)K_{LL}+\tfrac{1}{\beta_nr-1} K_{SS}\succeq K_{SL}+K_{LS}. 
\end{equation}
Finally note that,  $K_{SS}\preceq  \frac{1}{r}C_n\preceq \frac{1}{r}\nu_n\Rtilde_n\preceq \frac{1}{r}\nu_n(\beta^* r-1)D_S$, so 
\begin{align*}
\Kalman(\Chat_n)=K_{LL}+K_{LS}+K_{SL}+K_{SS}\preceq \beta_n rK_{LL}+\frac{\beta_n r}{\beta_n r-1}K_{SS}\preceq \beta_n \left(rK_{LL}+\frac{\nu_n(\beta^* r-1)}{r(\beta_n r-1)}D_S\right).
\end{align*}
Note that our choice of $\beta_n$ in \eqref{tmp:betachoice} makes the coefficient before $D_S$ less than $1$, so the proof is finished. 
\end{proof}
\begin{rem}
 Assumption \ref{aspt:cutoff} is not the direct replacement of Assumption \ref{aspt:working}, as we need an additional constant $\beta^*r-1$. This constant appears to be a necessary price to control the potential cross covariance between the two scales, which is achieved by a Cauchy Schwartz inequality \eqref{eqn:Cauchy}. In certain scenarios,  the cross covariance between two scales can be controlled by, say, localization structures, then the \eqref{eqn:Cauchy} is an overestimate, and $\beta^*r-1$ can probably be replaced by $1$.  In other words, there might be scenarios where Assumption \ref{aspt:working} holds while Assumption \ref{aspt:cutoff} does not. This is why we keep two assumptions in this paper instead of combining them. 
\end{rem}

\subsection{Filter error statistics}
In applications, other than the Mahalanobis error generated by the estimated covariances $C^L_n$ or $C^+_n$, there are other interesting error statistics: 1) The MSE $\E |e_n|^2$. 2) The Mahalanobis error generated by the optimal filter covariance $R_n$. This statistics shows a comparison between the reduced filter and the optimal filter, as the optimal filter error satisfies $\E\|X_n-m_n\|^2_{R_n}=d$. In many scenarios, we may find these error statistics  equivalent to the Mahalanobis error generated by $C^L_n$ or $C^+_n$. To see this, we can simply combine Theorems \ref{thm:dissmaha} and \ref{thm:cutoff},
\begin{cor}
\label{cor:useful}
Suppose system \eqref{sys:random} satisfies the reference projection Assumption \ref{aspt:cutoff}, then the Mahalanobis error of RKF generated by the reference covariance is bounded uniformly in time: 
\[
 \limsup_{n\to \infty}\E\|e_n\|^2_{\Rtilde_n^+}\leq \frac{2 dr}{1-\beta^*}. 
\]
If in addition the system noises are independent of the system coefficients, \eqref{eqn:independentnoise}, then 
\[
\E e_n\otimes e_n \preceq r\E \Rtilde_n+D_S.
\]  
As a consequence:
\begin{itemize}
\item Suppose that $\limsup\|\Rtilde_n^+\|\leq R$, then the MSE is bounded by $\limsup\E |e_n|^2\leq \frac{2 Rdr}{1-\beta^*}$. If in addition \eqref{eqn:independentnoise} holds, then 
$\E |e_n|^2\leq rRd$.  
\item 
Suppose that  $\Rtilde^+_n\preceq \rho^2 R_n$,  where $R_n$ is the covariance sequence of the  optimal filter \eqref{sys:optimal} and $\rho\geq 1$, then the performance of RKF is comparable with the optimal filter, as $\limsup\E \|e_n\|^2_{R_n}\leq  \frac{2 Rdr\rho^2}{1-\beta^*}$. If in addition \eqref{eqn:independentnoise} holds, then $\E e_n\otimes e_n\preceq \rho^2 r\E R_n$. 
\end{itemize}
\end{cor}
The requirements that $\|\Rtilde^+_n\|\leq R$ or $\Rtilde^+_n\preceq \rho^2 R_n$ can be verified by various ways discussed in Section \ref{sec:example}. As for DRKF, we consider only the MSE, because $\Rtilde^L_n$ is not directly comparable with $R_n$.
\begin{cor}
\label{cor:usefulD}
Suppose system \eqref{sys:random} is dynamically decoupled in two scales \eqref{eqn:blockdiag}, then if the Kalman filter covariance of system \eqref{sys:inflatedDRKF} satisfies $\Rtilde^L_0=\tfrac{1}{r}C^L_0$ and $\limsup\|\Rtilde^L_n\|\leq R$, the MSE of DRKF is bounded:
\[
\limsup_{n\to \infty} \E|e_n^L|^2\leq \frac{2Rp(1+\gamma_\sigma)}{r^2-r}+\frac{4\sqrt{\lambda_S}Rp\gamma_{\sigma}}{(r-\sqrt{r})(1-\sqrt{\lambda_S})}.
\]
 \end{cor}

\subsection{Reduced filter accuracy}
In many application scenarios, the observations are partial but very frequent and accurate. In such cases, one would expect the filter error to be small. This is quite easy to show for optimal Kalman filters, but not obvious for reduced filters. But with our framework, we can easily obtain the filter accuracy of the latter by the one of the former. 

\begin{cor}
Suppose system \eqref{sys:random} is dynamically decoupled in two scales with the stationary Kalman filter covariance of system \eqref{sys:inflatedDRKF} being bounded $\|\Rtilde^L_n\|\leq R$; or 
suppose the reference projection Assumption \ref{aspt:cutoff} holds with the stationary Kalman filter covariance of system \eqref{sys:inflated} being  bounded $\|\Rtilde^+_n\|\leq R$. In either case, assume  the stationary Kalman covariance attracts other Kalman filter covariance sequence as in   \cite{Bou93}. Then there is a DRKF, or RKF, for the following signal-observation system with small system and observation noises:
\begin{equation}
\label{eqn:scale}
\begin{gathered}
X^\epsilon_{n+1}=A_n X^\epsilon_n+B_n+\epsilon \xi_{n+1},\quad
Y^\epsilon_{n+1}=H_n X^\epsilon_{n+1}+\epsilon \zeta_{n+1}.
\end{gathered}
\end{equation}
The MSE of this filter scales like $\epsilon^2$. More precisely, there is a constant $D_R$ such that 
\[
\limsup_{n\to \infty}\E |e_n^\epsilon|^2\leq \epsilon^2 D_R.
\]
Here  $e^\epsilon_n$ stands for $X_n^{L,\epsilon}-\mu^{L,\epsilon}_n$ for DRKF, or $X_n^\epsilon-\mu^\epsilon_n$ for RKF. 
\end{cor}
\begin{proof}
In the dynamically decoupled case, the corresponding reference system will be 
\[
\begin{gathered}
X^{'L}_{n+1}=A^{'L}_n X^{'L}_n+B_n+\epsilon \xi_{n+1},\quad
Y^{'L}_{n+1}=H_n X^{'L}_{n+1}+\epsilon \zeta'_{n+1}.
\end{gathered}
\]
The stationary Kalman filter covariance of this system will be $\Rtilde^{L,\epsilon}_n=\epsilon^2 \Rtilde^L_n$, so 
\[
\limsup_{n\geq 0}\|\Rtilde^{L,\epsilon}_n\|=\epsilon^2\limsup_{n\geq 0}\|\Rtilde^L_n\|\leq \epsilon^2 R.
\] Then applying Corollary \ref{cor:usefulD} we have our claim. 

As for the second case, we apply RKF with $D_S^\epsilon=\epsilon^2 D_S$. The corresponding reference will be 
\[
\begin{gathered}
X'_{n+1}=A'_n X'_n+B_n+\epsilon \xi_{n+1},\quad
Y'_{n+1}=H_n X'_{n+1}+\epsilon \zeta_{n+1}.
\end{gathered}
\]
The stationary solution of this system will be $\Rtilde^\epsilon_n=\epsilon^2 \Rtilde_n$, so 
\[
\limsup_{n\geq 0}\|\Rtilde^{\epsilon+}_n\|=\epsilon^2\limsup_{n\geq 0}\|\Rtilde^+_n\|\leq \epsilon^2 R.
\] Then applying Corollary  \ref{cor:useful} we have our claim. 
\end{proof}

\section{General stochastic sequence setting}
\label{sec:general}
In some challenging scenarios, the reference stationary covariance $\Rtilde_n$ may not be a bounded sequence, then Assumption \ref{aspt:cutoff} cannot be verified. But weaker results may be obtainable, and interestingly the proofs do not need much of a change.
The content of this section is not necessary for most parts of Section \ref{sec:example}, and can be skipped in the first reading. 

An assumption that is more general than Assumption \ref{aspt:working} would be requiring the truncation error converges to a sequence that is stable on average:
\begin{aspt}
\label{aspt:workingg}
Suppose there is a stochastic sequence $\beta_n^*$ with a finite adjustment time $n_0$ such that the sequence \eqref{eqn:beta} satisfies $\beta_n\leq \beta_n^*$  for all $n\geq n_0$. 
\end{aspt}
The generalization of Theorem \ref{thm:dissmaha} is
\begin{theorem}
\label{thm:errorg}
For any fixed inflation $r>1$, consider applying the RKF \eqref{sys:RKF} to  system \eqref{sys:random}.  Suppose the large scale truncation of the RKF satisfies Assumption \ref{aspt:workingg}, then for any fixed times $ n_0\leq n$, 
\begin{equation}
\label{eqn:betan}
\E (\beta^*_{n_0+1}\cdots \beta^*_n)^{-1} \|e_n\|_{C_n^+}^2\leq \E  \|e_{n_0}\|^2_{C^+_{n_0}}+  2d\E \sum_{k=n_0+1}^n(\beta^*_{n_0+1}\cdots \beta^*_k)^{-1}.
\end{equation}
\end{theorem}
\begin{proof}
First of all, notice that the inequality \eqref{eqn:diss} still holds, since it does not depend on Assumption \ref{aspt:working}. Then our claim  is simply an induction, because
\begin{align*}
\E (\beta^*_{n_0+1}\cdots \beta^*_{n+1})^{-1}\|e_{n+1}\|_{C_{n+1}^+}=
\E (\beta^*_{n_0+1}\cdots \beta^*_{n+1})^{-1}\E_n\|e_{n+1}\|_{C_{n+1}^+}
\leq \E (\beta^*_{n_0+1}\cdots \beta^*_n)^{-1}(\|e_n\|^2_{C_n^+}+2d).
\end{align*}
If \eqref{eqn:betan} holds for time $n$ and we replace $\|e_n\|^2_{C_n^+}$ by its upperbound, then \eqref{eqn:betan} holds also for time $n+1$. 
\end{proof}
In order to verify the general Assumption \ref{aspt:workingg}, an a priori condition can also be derived from the reference Kalman covariance. 
\begin{aspt}
\label{aspt:cutoffg}
Let $\Rtilde_n$ be a (stationary) PD solution of \eqref{sys:controller}. Assume its small scale part is bounded as below with a stochastic sequence $\beta_n^*$
\[
\bfP_S\Rtilde_n\bfP_S \preceq (\beta_n^* r-1)D_S.  
\]
\end{aspt}

Since in the proof of Theorem \ref{thm:cutoff}, we used nothing about the fact that $\beta^*$ is a constant, so if we replace $\beta^*$ with $\beta^*_n$  in that proof, it is still valid. Therefore the following claim holds:
\begin{theorem}
\label{thm:cutoffg}
Suppose the general referenced projection Assumption \ref{aspt:cutoffg} holds, then Assumption \ref{aspt:workingg} also holds, and the acceptable reduction phase starts no later than 
\[
n_0=\lceil \log (\|\Rtilde_0^{-1}C_0\|)/\log(r'/r)\rceil.
\]
Moreover the covariance estimator is bounded by $C^+_{n}\preceq r \Rtilde_n+D_S$ for $n\geq n_0.$
\end{theorem}
\begin{rem}
\label{rem:general}
The previous discussion provides an easy generalization of our framework, but admittedly it buries some difficulties inside the result \eqref{eqn:betan}. If we want  Theorem \ref{thm:errorg} to provide concrete Mahalanobis error dissipation and convergence like in Theorem \ref{thm:dissmaha}, we roughly need to show
\begin{itemize}
\item $\E (\beta_{n_0+1}^*\cdots \beta^*_{n+1})^{-1}\geq \exp(b^*(n-n_0))$ for a constant $b^*>0$.
\item  $\E \sum_{k=n_0+1}^n(\beta_{n_0+1}^*\cdots \beta^*_{k})^{-1}\leq D\exp(b^*(n-n_0))$  for the same constant $b^*>0$, and some $D$.   
\end{itemize}
Usually it is not difficult to establish either of these ingredients, the major difficulty is that the growth ratio $b^*$ needs to be the same in both. Some special structures, like $\beta^*_k$ being independent of each other, will make the verification straightforward, but in general it is difficult. The authors also believe that \eqref{eqn:betan} may not be the best way to demonstrate the error dissipation in some scenarios, instead one should look for a Lyapunov function. But this is far away from the main theme of this paper, which is developing a general filter error analysis framework for large scale truncation. 
\end{rem}

\section{Applications and Examples}
\label{sec:example}
Given a concrete system \eqref{sys:random}, there might be various ways that the two-scale separation can be done. It is of practical importance to find the minimal large scale subspace, the proper inflation ratio $r$, while keeping the filter error small. Based on our previous results, these problems can be solved by numerically computing the Kalman filter covariance for the augmented system with a fixed  $r>1$, \eqref{sys:inflatedDRKF} or  \eqref{sys:inflated},  then verify Assumption \ref{aspt:cutoff} for RKF. The optimal two-scale separation and inflation can be obtained  by minimizing the MSE upper bound in Corollary \ref{cor:useful}. 

In this section, we will discuss a few general principles that may facilitate the filter error quantification and the verification of Assumption \ref{aspt:cutoff}, and how do they work in various dynamical scenarios. A simple stochastic turbulence model will be considered, and we will apply these principles to this model in different settings \cite{MH12}. 

\subsection{Some general guidelines for covariance bounds}
\label{sec:Kalmancor}
Section \ref{sec:offline} uses Kalman filters to provide a priori performance criteria. One of the advantages is that Kalman filters have a huge literature, so there are  many known results on how to  control  the Kalman filter covariance. We present in below a few simple ones. For the simplicity of illustration, we convey them only for system \eqref{sys:random} and its Kalman filter covariance $R_n$, while the same ideas are also applicable to the augmented systems \eqref{sys:inflatedDRKF}, \eqref{sys:inflated} and filter covariances $\Rtilde^L_n, \Rtilde_n$. 
\subsubsection{Unfiltered covariance}
\label{sec:uncondition}
In most applications, system \eqref{sys:random} has a stable dynamics itself, so the covariance of $X_n$ conditioned on the system coefficients $\mathcal{F}^c_n$ is  bounded uniformly in time. The computation of this covariance 
\[
V_n=\E_{\mathcal{F}^c_n} (X_n\otimes X_n)-\E_{\mathcal{F}^c_n} (X_n)\otimes \E_{\mathcal{F}^c_n}(X_n),
\]
follows a straightforward iteration: $V_{n+1}=A_nV_nA_n^T+\Sigma_n$, if it holds at $n=0$. In fact, we already used the small scale part $V^S_n$ for the formulation of DRKF. Then clearly $V_n\succeq R_n$.  Although this seems trivial, it is useful as it is independent of the choice of observations, and involves very little computation. 
\subsubsection{Equivalent transformation on observation}
\label{sec:generality}
Sometime changing the way we view the observations may simplify the computation by a lot. Mathematically speaking, we can 
consider a sequence of invertible $q\times q$ matrix $\Psi_n$, and the signal-observation system as below
\[
X_{n+1}=A_n X_n+B_n+\xi_{n+1},\quad
\widetilde{Y}_{n+1}=\Psi_n H_nX_{n+1}+\Psi_n\zeta_{n+1}.
\]
Intuitively, the Kalman filter performance of this system would be the same as \eqref{sys:random}. This is true, as one can check the Kalman covariance update operator $\Kalman$ is invariant under this transformation. This equivalent transformation can be used to simplify our notation. For example, we can let $\Psi_n=\sigma_n^{-1/2}$, then the observation noise for $\widetilde{Y}_{n}$ is a sequence of  i.i.d. Gaussian random variables. 

\subsubsection{Benchmark principle}
\label{sec:benchmark}
Since the Kalman filter \eqref{sys:optimal}  is the optimal filter for system \eqref{sys:random}, for any other estimator $\Xhat_n$ of $X_n$, its error  covariance is an upper bound for $R_n$:
\[
\E_{\mathcal{F}^c_n} (X_n-\Xhat_n)\otimes (X_n-\Xhat_n)=\E_{\mathcal{F}^c_n}\E_{\mathcal{F}^o_n} (X_n-\Xhat_n)\otimes (X_n-\Xhat_n)\succeq R_n. 
\]
So if there is an estimator $\Xhat_n$ with computable error covariance, we find a way to bound $R_n$. Although this idea is simple, it has been used many places to guarantee that $R_n$ is bounded, and as to the authors' knowledge, it is the only general strategy. The unfiltered covariance is actually  a special application of this principle, where the estimator is simply the mean, $\Xhat_n=\E_{\mathcal{F}^c_n}X_n$, which is updated through the recursion $\Xhat_{n+1}=A_n \Xhat_{n} +B_n$. 

When the observation $H_n$ is full rank, another simple estimator could be trusting the observation: $\Xhat_{n+1}=H^{-1}_nY_{n+1}$. The error covariance is  $[H_n^{T}]^{-1}\sigma_n H_n^T$. This idea can be generalized to the scenario where system \eqref{sys:inflated} is \emph{detectable} through a time interval $[m,n]$. Here we provide a simple and explicit estimate, while similar results can also be found in \cite{DP68, Jaz72, Sol96}. 
\begin{proposition}
\label{prop:Rr}
Denote the observability Gramian matrix as 
\[
\mathcal{O}_{n,m}=\sum_{k=m}^n A^{T}_{k,m}H_k^T\sigma_k^{-1}H_kA_{k,m},\quad A_{k,m}=A_{k-1}\cdots A_{j+1}A_{j}
\]
Suppose that $\Kalman_{n,m}=\mathcal{O}_{n,m}+\Rhat^{-1}_m$ is invertible, where $\Rhat_m$ is the prior covariance of $X_m$ without observing $Y_m$. Then
\[
R_n\preceq \sum_{j=m+1}^n Q_{n,m}^{j}\Sigma_j(Q_{n,m}^{j})^T+A_{n,m}\Kalman_{n,m}^{-1}A^{T}_{n,m},\quad Q^{j}_{n,m}=A_{n,m}\Kalman^{-1}_{n,m}\Kalman_{j,m}A_{j,m}^{-1}. 
\]
In case there is no prior knowledge of $X_m$, $\Rhat_m^{-1}$ can be set as a zero matrix, which is the inverse of the infinite covariance. 
\end{proposition}

\begin{proof}
For the simplicity of notations, in our proof, we do the general observation transformation, and replace $H_k$ by $\sigma_k^{-1/2}H_k$ and $\sigma_k$ by $I_q$. 
We will  first build up a  smoother for $X_{m}$ and then propagate it through time $[m,n]$. Also, without lost of generality, we assume $X_m\sim \mathcal{N}(0, \Rhat_m)$ and $B_k\equiv 0$.  Consider the estimator 
\[
\Xhat_{m}=\Kalman_{n,m}^{-1} \sum_{k=m}^{n}A^{T}_{k,m}H_{k-1}^TY_{k},\quad \Xhat_{n}=A_{n, m}\Xhat_m. 
\]
Notice that $X_k$ and $Y_k$ have the following moving average formulation:
\[
X_{k}=A_{k,m}X_m+\sum_{j=m+1}^kA_{k,j}\xi_j,\quad Y_{k}=H_{k-1}\left(A_{k,m}X_m+\sum_{j=m+1}^kA_{k,j}\xi_j\right)+\zeta_k.
\]
The error made by this estimator, $X_n-\Xhat_{n}$,  can be written as 
\begin{align*}
X_n-\Xhat_{n}=&\sum_{j=m}^{n}\left[A_{n,j}-A_{n,m}\Kalman^{-1}_{n,m}\sum_{k=j}^{n}
A^{T}_{k,m}H_{k-1}^TH_{k-1}A_{k,j}\right ]\xi_j-A_{n,m}\Kalman^{-1}_{n,m}\sum_{k=m+1}^{n}A^{T}_{n,k}H_{k-1}^T\zeta_k\\
\end{align*}
with $\xi_m=X_m$. When $\Rhat_m^{-1}=0$, one can check that the quantity above is independent of $X_m$. 

 Note that $\Kalman_{n,m}=A_{j,m}^{T}\mathcal{O}_{n,j}A_{j,m}+\Kalman_{j,m}$
\begin{align*}
A_{n,j}-&A_{n,m}\Kalman^{-1}_{n,m}\sum_{k=j}^{n}
A^{T}_{k,m}H_{k-1}^TH_{k-1}A_{k,j}
=A_{n,j}-A_{n,m}\Kalman^{-1}_{n,m}
A^{T}_{j,m}\mathcal{O}_{n,j}\\
&=A_{n,m}[I-\Kalman^{-1}_{n,m}
A^{T}_{j,m}\mathcal{O}_{n,j}A_{j,m}]A^{-1}_{j,m}
=A_{n,m}\Kalman^{-1}_{n,m}\Kalman_{j,m}A_{j,m}^{-1}=Q_{n,m}^{j}.
\end{align*}
In particular $Q_{n,m}^m=A_{n,m}\Kalman_{n,m}^{-1}R_m^{-1}$.
The expected error covariance $\E_{\mathcal{F}^c_n} (X_n-\Xhat_{n})\otimes (X_n-\Xhat_{n})$ will be bounded by 
\begin{align*}
&\sum_{j=m}^n Q_{n,m}^{j}\Sigma_j(Q_{n,m}^{j})^T+A_{n,m}\Kalman^{-1}_{n,m}\left(R_m^{-1}+\sum_{k=m+1}^nA^T_{n,k}H_{k-1}^T H_{k-1}A_{n,k}\right)\Kalman^{-1}_{n,m}A^{T}_{n,m}\\
&=\sum_{j=m}^n Q_{n,m}^{j}\Sigma_j(Q_{n,m}^{j})^T+A_{n,m}\Kalman^{-1}_{n,m}A^{T}_{n,m}.
\end{align*}
\end{proof}

\subsubsection{Comparison principles of Riccati equation}
In order to control $R_n$, sometimes it suffices to find another set of system coefficients, such that its Kalman filter covariance $R'_n\succeq R_n$. One way to generate such $R'_n$ is applying the comparison principle of Riccati equations for the forecast covariance \cite{FJ96}.
\begin{theorem}[Freiling and Jank 96]
\label{thm:FJ96}
Consider a signal-observation system
\[
\begin{gathered}
X_{n+1}'=A_n' X_n'+B_n' +\xi'_{n+1},\quad
Y_{n+1}'=H_n'X_{n+1}'+\zeta'_{n+1},
\end{gathered}
\]
with $\xi'_{n+1}\sim\mathcal{N}(0,\Sigma'_n)$ and $\zeta'_{n+1}\sim \mathcal{N}(0,\sigma'_n)$. Suppose the following holds a.s. with system coefficients of \eqref{sys:random}
\begin{equation}
\label{eqn:compare}
\begin{bmatrix}
\Sigma_n & A_n^T\\
A_n & -H_n^T \sigma_n^{-1}H_n 
\end{bmatrix}
\preceq
\begin{bmatrix}
\Sigma'_n & A_n^{'T}\\
A_n' & -H_n^{'T} \sigma_n^{'-1}H_n^{T}
\end{bmatrix}.
\end{equation}
Then if the forecast covariance satisfies $\Rhat_1\preceq \Rhat'_1$, we have $\Rhat_n\preceq \Rhat_n'$ for all $n\geq 1$. 
\end{theorem}
%As an application in our context, when we try to verify Assumption \ref{aspt:cutoff}, we can instead compute the stationary solution of a different inflated system with $(A_n^2, B_n^2, H_n^2, \Sigma_n^2)$, such that  the computation is easier because of sparse, diagonal, or any other special structures, while   
%\eqref{eqn:compare} holds with $(A_n^1, B_n^1, H_n^1, \Sigma_n^1)=(A_n', B_n', H_n', \Sigma_n')$. Then because 
%\[
%\Rhat^2_n\succeq \Rhat'_n\succeq \Kalman(\Rhat'_n)=\Rtilde_{n+1} \quad \text{ and } \Kalman(\Rhat^2_n)\succeq \Kalman(\Rhat'_n)
%\]
%it suffices to verify Assumption \ref{aspt:cutoff} for $\Rhat^2_n$ or $\Kalman(\Rhat^2_n)$. 
In particular, we can compare the reference Kalman filter of \eqref{sys:inflated} with the optimal filter \eqref{sys:optimal}:
\begin{proposition}
\label{prop:RRtilde}
Suppose that there are constants  $c$ and $C$ such that $c\Sigma_n\succeq  A_nD_S A_n^T$ and $A_n \Sigma_n^{-1}A_n^T\preceq C H_n\sigma_n^{-1} H_n^T $, and there is  a $\rho\geq 1$ such that 
\[
\frac{1}{\sigma}\left(1-\frac{1}{\rho^2}\right)\geq \frac{C(1-\sqrt{r'})^2}{\rho^2-r'(1+c)}.
\]
Then the stationary solution $\Rtilde_n$ of \eqref{sys:controller} is bounded by the stationary Kalman filter covariance $R_n$ of \eqref{sys:optimal} by the following
\[
\Rtilde_n\preceq \rho^2R_n.
\]
It is worth noticing that if $r'$ is close to $1$ and $c$ is close to $0$,  $\rho$ can be close to $1$ as well. 
\end{proposition}
\begin{proof}
We apply the equivalent observation transformation mentioned in Section \ref{sec:generality}, and assume $\sigma_n=I_q$. Let us consider the following inflation of \eqref{sys:random} with $\rho\geq 1$
\begin{equation}
\label{tmp:sysrho}
X^\rho_{n+1}=A_n X^\rho_n+B_n+\rho \xi_{n+1},\quad
Y^\rho_{n+1}=H_n X^\rho_{n+1}+\rho \zeta_{n+1}.
\end{equation}
Let $R^\rho_n$ be the  stationary filter covariance sequence of the associated Kalman filter, and $R_n$
be the one for \eqref{sys:random}. Evidently, the stationary solution of this system satisfies $R^\rho_n=\rho^2 R_n$, and so are the forecast covariances $\Rhat_{n}^\rho=\rho^2 \Rhat_{n}$. In order to  apply Theorem \ref{thm:FJ96} to the previous system and \eqref{sys:inflated}, we consider the following matrix difference
\begin{align*}
\begin{bmatrix}
\rho^2\Sigma_n & (A_n)^T\\
A_n & -\frac{1}{\rho^2\sigma}H_n^T H_n 
\end{bmatrix}-
\begin{bmatrix}
\Sigma_n' & (A_n')^T\\
A_n' & -\frac{1}{\sigma}H_n^T H_n 
\end{bmatrix}
&=
\begin{bmatrix}
(\rho^2-r')\Sigma_n-r' A_nD_SA_n^T & (1-\sqrt{r'})A_n^T\\
(1-\sqrt{r'})A_n & (\frac{1}{\sigma}-\frac{1}{\sigma\rho^2})H_n^T H_n 
\end{bmatrix}\\
&\preceq
\begin{bmatrix}
(\rho^2-(1+c)r')\Sigma_n & (1-\sqrt{r'})A_n^T\\
(1-\sqrt{r'})A_n & (\frac{1}{\sigma}-\frac{1}{\sigma\rho^2})H_n^T H_n 
\end{bmatrix}.
\end{align*}
With the conditions in the proposition, the matrix above is PSD. Therefore $\Rhat^\rho_n\succeq \widehat{\Rtilde}_n$, then because $\rho>1$ stands for a worse observation, it is straight forward to verify that
\[
\Rtilde_n=\Kalman(\widehat{\Rtilde}_n)\preceq \Kalman_\rho(\widehat{\Rtilde}_n)\preceq \Kalman_\rho (\Rhat^\rho_n)=\rho^2\Kalman(\rho^{-2}\Rhat^\rho_n)= \rho^2\Kalman(\Rhat_n)=\rho^2 R_n. 
\]
Here $K_\rho$ denotes the forecast-posterior Kalman covariance update for the  system \eqref{tmp:sysrho}.
\end{proof}

\subsection{Different settings}
The analysis framework of this paper can address system \eqref{sys:random} with very general setups. Meanwhile in applications,  particular dynamical and observation settings may require simplified computation or verification. 
\subsubsection{Classical setting}
In the classical setting, the system coefficients are deterministic and time homogenous, in other words they are of constant values. In this case, the stationary Kalman filter covariance matrices are also constant  $\Rtilde_n^L=\Rtilde^L, \Rtilde_n=\Rtilde$. Each of them solves an algebraic Riccati equation(ARE) equation 
\begin{equation}
\label{eqn:ARE}
\begin{gathered}
\Rtilde=\Kalman(\Rhat)=\Rhat-\Rhat H^T(\sigma+H \Rhat H^T)^{-1} H\Rhat,\quad \Rhat=r'A\Rtilde A^T+r'AD_S A^T+r'\Sigma,\\
\Rtilde^L=\Kalman_L(\Rhat^L),\quad \Rhat^L= rA^L\Rtilde^L (A^L)^T+\Sigma^L. \\
\end{gathered}
\end{equation}
In general, the solution require numerical methods to compute.

\subsubsection{Intermittent dynamical regimes}
One challenge that practical filters often face is that the dynamical coefficient $A_n$ is not always stable with spectral norm less than $1$. This is usually caused by the large scale chaotic dynamical regime transitions. One simple way of modeling this phenomenon, is letting $A_n$ be a Markov jump process on two states $\{A_+,A_-\}$, where $\|A_-\|\leq 1$ and $\|A_+\|>1$. Chapter 8 of \cite{MH12} has shown that this model could generate intermittent turbulence signals as seen in nature. Chapter 8 of \cite{MH12} has also numerically tested  the DRKF for the related filtering problem,  showing close to optimal performance. 

Our analysis framework naturally applies to these scenarios. The only difficulty  is that Assumption \ref{aspt:workingg} may require  additional works to verify. In general, one may need the general results of Section \ref{sec:general} or even other mechanisms mentioned in Remark \ref{rem:general}. 

On the other hand, in many practical situations, the random  dynamical regime switchings occur only on part of the model. If the large scale subspace includes this random part as in \cite{MH12}, the coefficients for small scale part are deterministic. This may make the conditions for  verification of our theorems the same as the deterministic case. For example, the formulation of Theorem \ref{thm:errorDRKF} for DRKF is independent of the large scale coefficients. For another example, if the large scale variables have no impact on the small scales, $\bfP_S A_n\bfP_L\equiv 0$, then when computing the unfiltered covariance for small scale $V^S$, the large scale coefficients also play no role. 

\subsubsection{Conditional Gaussian systems}
If the system coefficients are functions of the observation, that is $A_n=A(Y_n)$ and likewise for other terms, system \eqref{sys:random} is a conditional Gaussian system. Although the evolution of $(X_n, Y_n)$ in this case can be very nonlinear, the optimal filter is still \eqref{sys:optimal} according to \cite{LS01}. Such structure rises in many practical situations, like Lagrangian data assimilation, and turbulent diffusion with a mean flow.  The conditional Gaussian structure can be exploited in these situations to gain significant advantages \cite{CMT14, MT15}. In particular, dynamical structures like geostrophic balance can yield other types of reduced filters \cite{CMT14b}. 

In our context of reduced filtering, one caveat of conditional Gaussian system is that the system noises are in general not independent of the future system coefficients. For example, $Y_{n+1}$ may depend on $\xi_n$, and so does $A_{n+1}=A(Y_{n+1})$. As a consequence, Theorems \ref{thm:deterministic} and \ref{thm:errorDRKF} may not apply, while Theorem \ref{thm:dissmaha} still does. 

\subsubsection{Intermittent observations}
Due to equipment problems, observations sometimes are not available at each time step, but come in randomly.  \cite{Sin04}  models this feature by letting $H_n=\gamma_n H$ where $\gamma_n$ is a sequence of independent Bernoulli random variables with $\E\gamma_n=\bar{\gamma}$.  When the signal dynamics is unstable, \cite{Sin04} has shown that there is a critical frequency $\gamma_c$, such that the average Kalman filter covariance $\E R_n$ has a time uniform upper bound if and only if $\bar{\gamma}<\gamma_c$. Such results can be directly applied to the reference Kalman filters of systems \eqref{sys:inflatedDRKF} and \eqref{sys:inflated},  which leads to upper bounds for the reduced filter errors. On the other hand, if the system dynamics is stable, the reduced filter error can also be bounded using methods of Section \ref{sec:general}. This will be discussed in Section \ref{sec:intobs}.

\subsection{Stochastic turbulence examples}
One of the most important applications of filtering is on atmosphere and ocean. These are challenging problems as the system dimensions are extremely high, and the system parameters are changing constantly. One simple way to model the planetary turbulence flows is linearizing the stochastic dynamics in the Fourier domain. In order to apply the reduced filters to these models, we are interested in finding the minimal amount of Fourier modes for the large scale subspace, and how to set up the small scale covariance $D_S$ for RKF.

\subsubsection{Linearized stochastic turbulence in Fourier domain}
\label{sec:linearflow}
Consider the following stochastic partial differential equation \cite{HM08non,MH12}
\begin{equation}
\label{sys:SPDE}
\partial_t u(x,t)=\Omega(\partial_x) u(x,t)-\gamma(\partial_x) u(x,t)+F(x,t)+dW(x,t).
\end{equation}
For the simplicity of discussion,  the underlying space is assumed to be an one dimensional torus $\mathbb{T}=[0,2\pi]$, while generalization to higher dimensions is quite straight forward. The terms in \eqref{sys:SPDE} have the following physical interpretations: 
\begin{enumerate}[1)]
\item $\Omega$ is an odd polynomial of $\partial_x$. This term usually arises from the Coriolis effect from  earth's rotation, or the advection by another turbulence flow. 
\item $\gamma$ is a positive and even polynomial of $\partial_x$. This term models the general diffusion and damping of turbulences.
\item $F(x,t)$ is a deterministic forcing and $W(x,t)$ is a stochastic forcing. 
\end{enumerate}
In this paper, we assume both  forcing have a Fourier decomposition
\[
F(x,t)=\sum_{k\in \Index} f_k(t) e^{\rmi k\cdot x},\quad W(x,t)=\sum_k \sigma^u_k W_k(t) e^{\rmi k\cdot x}.
\]
Here $W_k(t)=\frac{1}{\sqrt{2}}W_{k,r}(t)+\frac{\rmi}{\sqrt{2}}W_{k,i}(t)$ is a standard Wiener process on $\mathbb{C}$, and the conjugacy condition is imposed to ensure terms in \eqref{sys:SPDE} are of real values: $f_k(t)=f^*_{-k}(t), \sigma^u_k=(\sigma^u_{-k})^*, W_k(t)=W_{-k}^*(t). $
 Suppose $P(\partial_x) e^{\rmi k\cdot x}= i\omega_k e^{\rmi k\cdot x}, \gamma(\partial_x) e^{\rmi k\cdot x}=\gamma_k e^{\rmi k\cdot x}$ with $\gamma_k>0$.  Then the solution of \eqref{sys:SPDE} can be written in terms of its Fourier coefficients, $u(x,t)=\sum_k u_ke^{\rmi k\cdot x}$, where the real and imaginary parts follow
\begin{equation}
\label{sys:fourier}
d\begin{bmatrix}u^r_k(t)\\
u^i_k(t)
\end{bmatrix}
=\begin{bmatrix}
-\gamma_k &-\omega_k\\
\omega_k &-\gamma_k
\end{bmatrix}
\begin{bmatrix}u^r_k(t)\\
u^i_k(t)
\end{bmatrix}
 dt+\begin{bmatrix}f^r_k(t)\\
f^i_k(t)
\end{bmatrix}dt+\frac{\sigma^u_k}{\sqrt{2}} \begin{bmatrix} dW^r_k(t)\\
dW^i_k(t)
\end{bmatrix}.
\end{equation}
To transform \eqref{sys:SPDE} to a discrete time formulation like \eqref{sys:random} in real domain, we assume the intervals between observations are of constant length $h>0$,  and pick a Galerkin truncation range $K\in \mathbb{N}$.  Let $X_n$ in \eqref{sys:random} be a $(2K+1)$-dim column vector, with coordinates being:
\begin{equation}
\label{eqn:Xn}
[X_n]_0=u_0(nh),\quad[X_n]_k=u^r_k(nh),\quad [X_n]_{-k}=u^i_k(nh),\quad k=1,\ldots, K.
\end{equation}
The system coefficients for the dynamic part of \eqref{sys:random} then can be formulated as follows, where $A_n=A$ is diagonal with $2\times 2$ sub-blocks, and $\Sigma_n=\Sigma$ is diagonal. Their entries are given below:
\begin{equation}
\label{sys:turb}
\begin{gathered}
\left[A\right]_{\{k,-k\}^2}=\exp(-\gamma_k h)\begin{bmatrix} \cos(\omega_k h) &\sin(\omega_k h)\\
-\sin(\omega_k h) & \cos(\omega_k h)
\end{bmatrix},\quad [B]_k=f^r_k(nh)h,\quad [B]_{-k}=f^i_k(nh)h,\\
[\Sigma]_{k,k}=\frac{(\sigma_k^u)^2}{2}\int_{nh}^{(n+1)h} \exp(-2\gamma_k s)ds=\frac{1}{2} E_k^u(1-\exp(-2\gamma_k h)).
\end{gathered}
\end{equation}
$E_k^u=\frac{1}{2\gamma_k}(\sigma_k^u)^2$ stands for the  stochastic energy of the $k$-th Fourier mode, and also the sum of stochastic energy of $[X_n]_k$ and $[X_n]_{-k}$.

In practice, the damping often grows and the energy decays like polynomials of the wavenumber $|k|$
\begin{equation}
\label{eqn:physic}
 \gamma_k=\gamma_0+\nu |k|^\alpha,\quad E_k^u=E_0|k|^{-\beta},\quad \alpha> 0,\beta\geq 0.
\end{equation}
As we will see in our discussion below, such formulation guarantees the existence of a large scale separation with good reduced filter performance. To show that our framework is directly computable, we will also consider the following specific set of physical parameters with a Kolmogorov energy spectrum used in \cite{MG13}:
\begin{equation}
\label{eqn:physics}
\alpha=2,\quad \beta=\frac{5}{3},\quad r=1.2,\quad r'=1.21, \quad h=0.1,\quad \nu=0.01, \quad \beta^*=0.9,\quad E_0=1.
\end{equation}

\subsubsection{Setups for reduced filters}
\label{sec:turbreduce}
Since the system coefficients of \eqref{sys:turb} are all block diagonal, both DRKF and RKF can be applied for reduced filters. Naturally, the large scale set consists of modes with wavenumbers $\{|k|< N\}$. And for RKF, $D_S$ should be a diagonal matrix with entries $\{\delta_k\}_{|k|\geq N}$. The question is how to pick these reduced filter parameters, and how do they depend on the system coefficients. 
\par 
DRKF does not have additional constraint, as Theorem \ref{thm:errorDRKF} always provide an upper bound. But in order to have good practical performances, intuitively the error caused by small scale time correlation should be of scale $\epsilon$ comparing with the other terms. In other words, 
\begin{equation}
\label{eqn:DRKFerrorrequire}
\frac{2\sqrt{\lambda_S r}(\sqrt{r}+1)\gamma_\sigma}{(1-\sqrt{\lambda_S})(1+\gamma_\sigma)}\leq \epsilon, 
\end{equation}
$\lambda_S$ in our setting will be $\max_{|k|\geq N}\exp(-\gamma_k h)=\exp(-\gamma_N h)$. If we approximate $(1-\sqrt{\lambda_S})$ with $1$, and bound $\gamma_\sigma$ with $1$, we find that $
\gamma_N\geq -\frac{2}{h}\log (\epsilon/\sqrt{r(r+1)}).$
This relation is  independent of the energy spectrum, and if the dissipation has a polynomial growth \eqref{eqn:physic}, we find that 
\[
N\geq [-\frac{2}{h\nu}\log (\epsilon/\sqrt{r(r+1)})]^{\frac{1}{a}}.
\] 
In the physical setup of  \eqref{eqn:physics} with $\epsilon=0.2$, we find that $N\approx  65$. 

RKF requires the verification of Assumption \ref{aspt:cutoff}. Here we uses the unfiltered covariance $\widetilde{V}$ of the inflated system \eqref{sys:inflated} as an upper bound for $\Rtilde_n$. This applies independently of the observation setup. It is easy to find that  $\widetilde{V}$  is  diagonal with entries 
\begin{equation}
\label{eqn:Vk}
[\widetilde{V}]_{k,k}=\tilde{v}_k=\frac{r'E_k^u(1-r'\exp(-2\gamma_k h))+\delta_k r'\exp(-2\gamma_k h)}{2-2r'\exp(-2\gamma_k h)}.
\end{equation}
In order for Assumption \ref{aspt:cutoff} to hold, we need that for some $\beta^*\leq 1$
\begin{equation}
\label{tmp:deltak}
\widetilde{v}_k\leq  (\beta^* r-1) \delta_k \quad k\geq N.
\end{equation}
In order to achieve this,  we need $\beta^*r\geq \beta^* rr'\exp(-2\gamma_k h)+1$ and 
$\delta_k\geq \frac{r' E_k^u}{\beta^*r-\beta^* rr'\exp(-2\gamma_k h)-1}.$
In the setting of \eqref{eqn:physic}, $\exp(-2\gamma_k h)\to 0$  for large $|k|$, so we roughly require 
\begin{equation}
\label{tmp:Ds}
\delta_k\geq \frac{r' E^u_k}{\beta^* r-1} \quad \Rightarrow \quad D_S\approx \frac{r'\bfP_SE^u\bfP_S}{\beta^* r-1}. 
\end{equation}
 The small scale truncation requires $
\gamma_N \geq  \frac{1}{2h}\log \left(\frac{1}{r'}-\frac{1}{\beta^* r r'}\right)$.
In the polynomial dissipation setting \eqref{eqn:physic}, this implies $N\geq [\frac{1}{2h\nu }\log \left(\frac{1}{r'}-\frac{1}{\beta^* r r'}\right)]^{\frac{1}{\alpha}}$. In particular with the physical parameters \eqref{eqn:physics}, $N\approx 25$. 

\subsubsection{Intermittent physical environment}
A simple way to model intermittent physical environment for stochastic turbulence \eqref{sys:turb} is letting $A_n$ be a Markov chain, while maintaining the sub-block structure: $[A_n]_{\{k,-k\}^2}=[\lambda_{n}]_k[A]_{\{k,-k\}^2}.$
Here $\lambda_n$ is a Markov chain taking values in $\reals^{K+1}$. Then the system random instability can be modeled as the random fluctuation of $[\lambda_n]_k$, so that occasionally $\|[A_n]_{\{k,-k\}^2}\|>1$ for some $k$. 

In many situations, such instability only occur on the a small subset $I$ of Fourier modes. This is because when the wave numbers are high,  the dissipation force is much stronger than the random environmental forcing. So for $k\in I^c$, $[A_n]_{\{k,-k\}}$ could remain of constant value like in \eqref{sys:turb}. Then it suffices to let the large scale mode set include  subset $I$, and the discussion of Section \ref{sec:turbreduce} remains the same. This idea also  applies to systems with random coefficients on all modes as well, as long as another system with constant small scale coefficients exists as an upper bound in the sense of 
Theorem \ref{thm:FJ96}. 

\subsubsection{Advection from a strong jet flow}
One major nonlinearity source for planetary or engineering turbulence takes the form of a jet flow advection. For example, the  meridional flows  on earth are often advected by a eastward zonal flow \cite{MG13, MT15}. \eqref{sys:SPDE} can be extended to this scenario, by adding an auxiliary process $w_t\in \reals$ to describe the jet flow, with $B_t$ being an independent standard Wiener process in $\reals$,
\begin{equation}
\label{sys:jet}
\begin{gathered}
dw_t=G_{w_t}(u_t)dt+g_t dt+\sigma_w dB_t,\\
\partial_t u(x,t)=(\Omega(\partial_x)+ w_t\partial_x -\gamma(\partial_x)) u(x,t)dt  +F(x,t)dt+dW(x,t).
\end{gathered}
\end{equation}
The feedback  of $u_t=u(\,\cdot\,,t)$ on $w_t$, $G_{w_t}(u_t)$, is assumed to be linear on $u_t$, but may have nonlinear dependence on $w_t$. Since strong jet flows often have close to accurate observations, we assume $w_t$ is part of the  observation. The resulting system will be conditionally Gaussian. A time discretization like in Section \ref{sec:linearflow} would lead to the same dynamical formulation as \eqref{sys:turb}, except that the phase speed $\omega_k$ is replaced by $\omega_k+kW_n$. $W_n=w_{nh}$ is the time discretization of $w_t$, and follows 
\[
W_{n+1}=\widetilde{G}_{W_n}X_{n+1}h+W_n+g_{nh} h+\sigma_w\sqrt{h}\zeta^v_{n+1}.
\]
Here $\widetilde{G}_w X=G_w u$, if $X$ consists of the Fourier modes of field $u$ like in \eqref{eqn:Xn}. $W_{n+1}$ can be seen as the $q+1$-th dimension of the observation vector $Y_{n+1}$, and $\zeta^v_{n+1}$ is its observation noise. This makes the time discretized model in the form of \eqref{sys:random}. 
\par
Jet flow advection in fact is a good example to show that system independent noise condition \eqref{eqn:independentnoise} may fail, since the observation noise $\zeta^v_{n+1}$ is correlated with coefficients $A_{n+1}$ and $H_{n+1}$ through $W_{n+1}$. As a consequence, Theorems \ref{thm:deterministic} and \ref{thm:errorDRKF} no longer apply, but Theorem \ref{thm:dissmaha} still holds.

%\subsubsection{Regularly spaced observations}
%So our analysis can be done independently over these aliasing set, as \cite{MH12} suggested.
%
%One special case is that $J=K$. Then each aliasing set contains a single point, so there is an direct observation of each mode $u_k$ with no aliasing effect. 
%Then for each fixed wavenumber $k$, the associated ARE is an diagonal equation. Assuming the solution is diagonal with entries $r_k$, we find that 
%\[
%r_k=w\hat{r}_k(w+\hat{r}_k)^{-1},\quad
%\hat{r}_k=a_k r_k+w_k, 
%\]
%where
%\[
%w=\frac{\sigma^o}{2J+1},\quad a_k=r'\exp(-2\gamma_k h),\quad w_k=E_k^u(1-\exp(-2\gamma_k h))+\exp(-2\gamma_kh)\delta_k. 
%\]
%The solution bears the form
%\[
%r_k=\frac{1}{2}a_k^{-1}(a_k w-w_k-w+\sqrt{(a_k w-w_k-w)^2+4w_kw}).
%\]
%In order for Assumption \ref{aspt:cutoff} to hold, we need 
%\[
%r_k\leq \delta_k(\beta^*r-1)\quad |k|\geq N.
%\]
%or equivalently with $b_k=\exp(-2\gamma_k h)$
%\[
%a_k(\beta^*r-1)^2\delta_k^2+(\beta^*r-1)\delta_k(b_k\delta_k+(1-b_k)E_k^u-a_k w+w)\leq w(b_k\delta_k+(1-b_k)E_k^u).
%\]
 
\subsubsection{Intermittent observations}
\label{sec:intobs}
Observations of turbulence in practice often come from a network of sensors, that are located at a group of points $x_j\in \mathbb{T}$, and the observation noise can be modeled by i.i.d. $\mathcal{N}(0, \sigma)$ random variables: 
\begin{equation}
\label{eqn:equalspace}
[H]_{j,0}=1, \quad [H]_{j,k}=2\cos(k x_j),\quad [H]_{j,-k}=2\sin(k x_j),\quad \sigma=\sigma^o I_q.  
\end{equation}
One particular choice of sensor location will be equally spacing, $x_j=\frac{2\pi j}{2J+1}, j=0,1,\ldots, 2J$, studied by chapter 7 of \cite{MH12}. Consider an equivalent observation transformation
\[
[\Psi]_{0,j}=\frac{1}{2J+1},\quad [\Psi]_{i,j}=\frac{\cos( \frac{2\pi i j}{2J+1})}{2J+1},\quad [\Psi]_{-i,j}=\frac{\sin( \frac{2\pi i j}{2J+1})}{2J+1},
\]
so the transformed observation coefficients satisfies $[\Psi H]_{j,k}=\delta_{j\equiv k\,\,mod\,\, 2J+1}$, and $\Psi \sigma \Psi^T=\frac{\sigma^o I_{2J+1}}{2J+1}.$ When $J<K$, such observation network introduces aliasing effect among the Fourier modes, which is carefully studied in \cite{HM08, MH12}. Here after we focus only on the simple case where $K=J$ so $\Psi H=I$. 

In real applications, turbulence observations may not be available at each time step. Following the example of \cite{Sin04}, we model this problem by letting  $H_n=\gamma_n H$, where  $\gamma_n$ is a sequence of Bernoulli random variables with average $\E\gamma_n=\bar{\gamma}$. We will look at how does such observation changes the reduced filter setup.

For DRKF, the small scale unfiltered  covariance $V$ is diagonal with entries $[V]_{k,k}=\frac{1}{2}E_k^u$. Then $\gamma_\sigma=\sup_n\|(\sigma^L_n)^{-1} H_nV^SH_n^T \|=\sup_{|k|\geq N}\frac{(2K+1)E_k^u}{(2K+1)E^u_k+2\sigma^o}$. Following the discussion in Section \ref{sec:turbreduce}, we are interested in maintaining \eqref{eqn:DRKFerrorrequire}. In the polynomial dissipation regime \eqref{eqn:physic}, if we approximate $1-\sqrt{\lambda_S}$ by $1$, and replace $1+\gamma_\sigma$ by a lower bound $1$, we find that
\[
\frac{\exp(-\frac{1}{2}h\nu N^\alpha)E_0N^{-\beta}}{E_0N^{-\beta}+\frac{2\sigma^o}{2K+1}}\leq \frac{\epsilon}{\sqrt{r(r+1)}}
\]
In the physical setup of \eqref{eqn:physics} with $\epsilon=0.2, \sigma^o=0.1, K=200$ and  by numerical computing the quantities above, we find that $N\approx 59$.

As for RKF, for any fixed time $n$,  the unfiltered covariance $\widetilde{V}$ is given by \eqref{eqn:Vk}, and we know the reference Kalman filter covariance $\Rtilde_n\preceq \widetilde{V}$.  If at time $n$, the observations are available, $\gamma_n=1$, note that $\widehat{\Rtilde}_n\preceq A\widetilde{V}A^T+\Sigma=\widetilde{V}$,  by Proposition \ref{prop:Rr} with $m=n$,
\[
[\Rtilde_n]_{k,k}\leq v'_k=\frac{\tilde{v}_k\sigma^o}{\sigma^o+(2K+1)\tilde{v}_k}<\tilde{v}_k.
\] 
Denote $ \beta_o=\max_{k\geq N} \frac{\tilde{v}_k}{r\delta_k}+\frac{1}{r}, \beta_u=\max_{k\geq N} \frac{\tilde{v}_k}{r\delta_k}+\frac{1}{r}.$
Clearly $\beta_o<\beta_u$. So in Assumption \ref{aspt:workingg}, we can let
\[
\beta_n^*=\gamma_n \beta_o+(1-\gamma_n)\beta_u,
\]
which is an independent sequence. In order for the general Theorem \ref{thm:errorg} to give a meaningful upper bound, it suffices to require 
\begin{equation}
\label{tmp:betabar}
\bar{\beta}^*=\E \beta_n^*=\bar{\gamma}\beta_o+(1-\bar{\gamma})\beta_u<1.
\end{equation}
With \eqref{tmp:betabar}, we will  have $\E \|e_n\|^2_{C_n}\leq \bar{\beta}^{*(n-n_0)} \E \|e_{n_0}\|^2_{C_{n_0}}+\frac{2d}{1-\bar{\beta}^*}$. Since \eqref{tmp:deltak} is equivalent to $\beta_u<1$, so \eqref{tmp:betabar} is a weaker requirement and end up with a smaller $N$. In particular, if we pick $D_S$ as in \eqref{tmp:Ds}, the parameters as in \eqref{eqn:physics}, and let $\bar{\gamma}=0.9, \sigma^o=0.1, K=200$, we find $N\approx 14$.  

\section{Conclusion and discussion}
High dimensionality is an important challenge for modern day numerical filtering, as the classical Kalman filter is no longer computationally feasible. This problem can sometime be resolved by proper dimension reduction techniques,  exploiting intrinsic multiscale structures. This paper considers two of such reduced filters. The DRKF works for dynamically decoupled systems, and estimates the small scale variables with their equilibrium statistical states. The RKF uses a constant statistical state for the small scale filtering prior, and requires the large scale projection not to decrease the error covariance. Both methods have been studied by \cite{MH12} for stochastic turbulence filtering, and they have close to optimal performances in various regimes. On the other hand, rigorous error analysis of these reduced filter has been an open problem, since the dimension reduction techniques bring in unavoidable biases, just like in many other practical uncertainty quantification procedures. This paper fills in this gap by developing a two-step framework. The first step examines the fidelity of the reduced covariance estimators, showing that the real filter error covariance is not underestimated. For RKF with system independent noises, this can be verified by tracking the covariance matrix. For DRKF and more general scenarios, the covariance fidelity can be demonstrated by the intrinsic dissipation mechanism of the Mahalanobis error. The second step shows how to bound the reduced filter covariance estimators, by building a connection between them and proper Kalman filter covariances. The combination of these two steps yields an error analysis framework for the reduced filters, with exponential stability and accuracy for small system noises as simple corollaries. When applied  to a linearized stochastic turbulence, this framework provides a priori guidelines for large scale projection range and reduced filter parameterizations. 

Besides the major themes mentioned above, there are two related issues we have  not focused on:
\begin{itemize}
\item The multiplicative inflation is applied in our reduced filters to avoid covariance underestimation. This technique has been applied widely for various practical filters, but its theoretical importance has never been studied except in one dimension \cite{FB07}. The error analysis of this paper implicitly studies this issue, as the inflation plays an important role in our proof. Based on the formulation of Theorems \ref{thm:dissmaha}, \ref{thm:errorDRKF}, and \ref{thm:cutoff}, stronger inflation provides better filter stability. Moreover, as mentioned in Remark \ref{rem:inflationRKF}, this inflation is an essential high dimension replacement of the classical uniform bounded conditions in \cite{RGYU99}. 
\item For RKF, Theorem \ref{thm:deterministic} has a much stronger result comparing with Theorem \ref{thm:dissmaha}, while the additional condition on system independent noises often holds. But for many other practical filters such as the ensemble Kalman filter, the second moment of the filter error is not traceable, as the Kalman gain matrix is correlated with the filter error. The Mahalanobis error dissipation on the other hand still holds as it is a more intrinsic property. 
\end{itemize} 

\section*{Acknowledgement}
This research is supported by the MURI award grant N00014-16-1-2161, where A.J.M. is the principal investigator, while X.T.T. is supported as a postdoctoral fellow. The author also thank Kim Chuan Toh for his discussion on Lemma \ref{lem:Kalconcave}.

\appendix
\section{Complexity estimates}
\label{sec:complexity}
In this section we do some simple computational complexity estimates for the Kalman filter \eqref{sys:optimal} and the reduced Kalman filters \eqref{sys:DRKF} and \eqref{sys:RKF}. Through these estimations, we find that the reduced filters reduce computation complexity from $O(d^2q)$ to $O(d^2+p^2q+p^3)$ and $O(d^2+dp^2+dq^2)$, which is a significant reduction when the state space dimension $d$ is much larger than the observation dimension $q$ and large scale dimension $p$. For simplicity, we only consider the most direct numerical implementation of the related formulas, although there are many alternative implementation methods that increases numerical stability and accuracy \cite{GA01}.  We assume the complexity of matrix product of $[A]_{a\times b}$ and $[B]_{b\times c}$ is $abc$, and the complexity of the inversion and Cholesky decomposition of a general $[A]_{a\times a}$ matrix is $a^3$ \cite{GV83}. There are also a few additional assumptions that hold for most applications, while without them similar qualitative claims hold as well.
\begin{enumerate}[1)]
\item We focus mostly on the online computational cost, which is the cost for the computation of filter iteration. This is the most significant cost in the long run. 
\item When the system coefficients are deterministic, the Kalman gain matrix sequence in principle can be computed offline \cite{GA01}. We do not  consider this scenario as it oversimplifies the  problem.
\item $A_n$ is a sequence of sparse matrices. This holds for the stochastic turbulence models in Section \ref{sec:example}. It rises in various differential equation context as most physical interaction involves only elements in close neighbors. As a consequence,  $\Chat_{n+1}=A_n C_n A_n^T+\Sigma_n$ involves only $O(d^2)$ complexity instead of $O(d^3)$. On the other hand, if this assumption is not true, then the leading computational cost is $O(d^3)$ and comes from the prescribed forecast step, while the reduced filters obviously reduce the cost to $O(p^2 d)$, so there is no need of further discussion. 
\item $H_n$ and $\sigma_n$ are  also  sparse matrices with relatively time invariant structure, so matrix product like $H_n \Chat_n H_n^T$ involves only $O(d^2)$ computation. This assumption holds as in many applications, the observations are over a few dimensions and the observation noises are independent.
\end{enumerate}
Based on these assumptions, the complexity of Kalman filter is given by Table \ref{tab:KF}.

\begin{table}[h!]
\begin{center}
\begin{tabular}{|c | c| }
\hline
 Operation & Complexity order \\
 \hline 
 $\Rhat_n=A_n R_{n+1} A_n^T+\Sigma_n$ & $d^2$\\
$ (\sigma_n+H_n \Rhat_n H^T_n)^{-1}$ & $d^2+q^3+d^2q$\\
 $K_{n+1}=\Rhat_n H^T_n (\sigma_n+H_n \Rhat_{n+1} H^T_n)^{-1}$ & $ d^2 q$\\
 $m_{n+1}=A_n m_n+B_n-K_{n+1}(Y_{n+1}-H_n (A_n m_n+B_n))$ & $d+dq$\\
$R_{n+1}=\Rhat_{n+1}-\Rhat_{n+1} H^T_n (\sigma_n+H_n \Rhat_{n+1} H^T_n)^{-1} H_n \Rhat_{n+1}$ & $d^2q$\\
\hline
total & $d^2 q$\\
\hline
\end{tabular}
\caption{Complexity estimate of the Kalman filter \eqref{sys:optimal}.}
\label{tab:KF}
\end{center}
\end{table}

\subsection{DRKF}
DRKF essentially is applying a Kalman filter in the large scale subspace with dimension $p$. The only additional computation involves estimating the unfiltered small scale covariance $V^S_n$ which involves $O(d^2)$ computation. When the system coefficients are constants, $V^S_n$ is of constant value and there is no need to update it. We put such savable cost in brackets in the Table \ref{tab:DRKF}.
\begin{table}[h!]
\begin{center}
\begin{tabular}{|c | c| }
\hline
 Operation & Complexity order \\
 \hline 
 $\Chat^L_{n+1}=A^L_n C^L_n (A^L_n)^T+\Sigma^L_n$ & $p^2$\\
 $ V^S_{n+1}=A_nV^S_n A^T_n+\Sigma^S_n, \,\,\,\mu^S_{n+1}=A^S_n \mu^S_n+B_n^S$ & $(d^2)$\\
 $ (\sigma^L_n+H_n \Chat^L_{n+1} H^T_n)^{-1}$ & $(d^2)+q^3$\\
 $K^L_{n+1}=\Chat^L_{n+1} (H^L_n)^T (\sigma^L_n+H_n^L \Chat^L_{n+1} (H^L_n)^T)^{-1}$ & $ p^2 q$\\
 $\mu^L_{n+1}=A^L_n \mu^L_n+B^L_n-K^L_{n+1}(Y_{n+1}-H^S_n \mu^S_{n+1}-H^L_n (A^L_n \mu^L_n+B^L_n))$ & $pq$\\
$C^L_{n+1}=r\Chat^L_{n+1}-r\Chat^L_{n+1} (H^L_n)^T (\sigma^L_n+H^L_n \Rhat^L_n (H^L_n)^T)^{-1} H^L_n \Chat_{n+1}^L$ & $q^3+p^2q$\\
\hline
total & $q^3+p^2q+(d^2)$\\
\hline
\end{tabular}
\caption{Complexity estimate of the DRKF \eqref{sys:DRKF}.}
\label{tab:DRKF}
\end{center}
\end{table}

\subsection{RKF}
In the implementation of  RKF, we need to exploit the fact that $C_n$ is nonzero only for the upper $p\times p$ sub-block. Therefore its Cholesky decomposition involves a cost of $O(p^3)$. Also one would like see  $\Chat_n$ as the sum of $A_n C_n A_n^T$, which is a rank $p$ matrix, and a sparse matrix $A_nD_S A_n^T+\Sigma_n$, instead of a generic $d\times d$ matrix. The Woodbury matrix identity is also useful for gaining computational advantage. For example, when doing the matrix inversion
\[
[\sigma_n+H_n \Chat_n H^T_n]^{-1}=[Q_n+H_n A_n C_n A_n^T H_n^T ]^{-1},
\] 
where  $Q_n:=[\sigma_n +H_n\Sigma'_nH^T_n]$ with $\Sigma_n'=\Sigma_n+D_S$, note that inverting $Q_n$ costs $O(q^3)$.  The Woodbury identity indicates that:
\[
[Q_n+H_n A_n C_n A_n^T H_n^T ]^{-1}=Q^{-1}_n-
Q^{-1}_nH_nA_n C_n^{1/2} [I+C_n^{1/2}A_n^TH_n^TH_nA_n C_n^{1/2}]^{-1}C_n^{1/2}H_n^TA_n^TQ^{-1}_n.
\]
Note that  $C_n^{1/2}[I+C_n^{1/2}A_n^TH_n^T H_nA_n C_n^{1/2}]^{-1}C_n^{1/2}$ has only the upper $p\times p$ sub-block being nonzero, so its computation costs only $O(p^3+pqd+p^2q)$. So the overall cost of computing $[\sigma_n+H_n \Chat_n H^T_n]^{-1}$ is $O(pqd+p^3+q^3)$, while in the Kalman filter, it is $qd^2$. The estimate of each step is given below in Table \ref{tab:RKF}.
\begin{table}[h!]
\begin{center}
\begin{tabular}{|c | c| }
\hline
 Operation & Complexity order \\
 \hline 
 $\Chat_n=A_n C_n A_n^T+A_n D_SA_n^T+\Sigma_n$ & $d^2$\\
$ (\sigma_n+H_n \Chat_n H^T_n)^{-1}$ & $pqd+q^3+p^3$\\
 $\Khat_{n+1}=A_n C_n A_n^T H^T_n (\sigma_n+H_n \Chat_n H^T_n)^{-1}$ & $p^2d+pqd$\\
 $\quad\quad+\Sigma'_n H_n^T(\sigma_n+H_n \Chat_n H^T_n)^{-1}$&$dq^2+d^2$\\
 $\mu_{n+1}=A_n \mu_n+B_n-\Khat_{n+1}(Y_n-H_n (A_n \mu_n+B_n))$ & $d+dq$\\
$C_{n+1}=r\bfP_L\Chat_n\bfP_L-r\bfP_L\Chat_n H^T_n (\sigma_n+H_n \Chat_n H^T_n)^{-1} H_n \Chat_n\bfP_L$ & $dp^2+dq^2$\\
\hline
total & $d^2+dp^2+dq^2$\\
\hline
\end{tabular}
\caption{Complexity estimate of the RKF \eqref{sys:RKF}.}
\label{tab:RKF}
\end{center}
\end{table}

\section{Matrix inequalities}
The following lemma has been mentioned in \cite{FB07} for dimension one. 
\begin{lem}
\label{lem:Kalconcave}
The prior-posterior Kalman covariance update mapping $\Kalman$ in \eqref{sys:optimal}, can also be defined as
\[
\Kalman(C)=(I-K H_n) C(I-K H_n)^T + K  \sigma K^T
\]
where $K:=C H_n (\sigma +H_n C H_n^T)^{-1}$ is the corresponding Kalman gain. $\Kalman$ is a  concave monotone operator from PD to itself. 
\end{lem}
\begin{proof} The first matrix identity is straightforward to verify, and can be found in many references of Kalman filters \cite{MH12}.
In order to simplify the notations, we let $H=H_n$ and  $J(X)=(\sigma +HXH^T)^{-1}$. Then picking any symmetric matrix $A$, the perturbation in direction $A$ is given by 
\[
D_A J(X):=\frac{d}{dt}J(X+At)\big|_{t=0}=-JHAH^TJ.
\]
Therefore
\[
D_A \Kalman=A-AH^TJHX-XH^TJHA+XH^TJHAH^TJHX=(I-H^TJHX)^TA(I-H^TJHX)
\]
The Hessian is 
\begin{align*}
D_A^2 \Kalman&=-2AH^TJHA+2AH^TJHAH^TJHX+2XH^TJHAH^TJHA\\
&\phantom{==}-2XH^TJHAH^TJHAH^TJHX\\
&=-2(AH^TJ^{1/2}-XH^THAH^TJ^{1/2})\cdot(AH^TJ^{1/2}-XH^THAH^TJ^{1/2})^T\preceq 0.
\end{align*}
Therefore, as long as $X, X+A\succeq 0$, then the convexity holds:
\[
\Kalman (X)+\Kalman(X+A)\preceq 2\Kalman(X+\tfrac12 A). 
\]
When we require $A$ to be PSD, $D_A\Kalman\succeq 0$ implies the monotonicity of $\Kalman$. 
\end{proof}

\begin{lem}
\label{lem:matrix}
Suppose that  $A, C, D$ are PSD matrices, $C$ is invertible,  while $A\preceq [BCB^T+D]^{-1}$, then
\[
B^TAB\preceq C^{-1},\quad A^{1/2}D A^{1/2}\preceq I_d. 
\]
\end{lem}
\begin{proof}
From the condition, we have $A^{1/2}[BCB^T+D] A^{1/2}\preceq I_d$. Therefore our second claim holds. Moreover, 
\[
(B^TAB)C(B^TAB)\preceq B^TA^{1/2} A^{1/2}[BCB^T+D] A^{1/2} A^{1/2}B\preceq B^TA B. 
\]
This leads to our first claim by the next lemma.
\end{proof}

\begin{lem}
\label{lem:trivial}
Let $A$ and $B$ be PSD matrices, if
\begin{itemize}
\item $A\succeq I_d$, then $ABA\succeq B$.
\item $A\preceq I_d$, then $ABA\preceq B$. And for any real symmetric matrix $C$, $CAC\preceq C^2$. 
 \end{itemize}
\end{lem}
\begin{proof}
If the null subspace of $B$ is $D$ and $\bfP$ is the projection onto the complementary subspace $D^\bot$, then it suffices to show that  $(\bfP A\bfP) (\bfP B\bfP) (\bfP A\bfP) \succeq \bfP B\bfP$. Therefore, without loss of generality, we can assume $B$ is invertible, so it suffices to show 
\[
(B^{-1/2} A B^{1/2})(B^{-1/2} A B^{1/2})^T\succeq I.
\]
But this is equivalent to checking the singular values of $B^{-1/2}A B^{1/2}$ are greater than $1$, which are the same as the  eigenvalues of $A$.

If $A$ and $C$ are invertible, then the second claim follows as the direct inverse of the first claim. Else, it suffice to show the claim on the subspace where $A$ and $C$ are invertible. 
\end{proof}
\begin{lem}
\label{lem:norm}
Let $A$ and $B$ be two PSD matrices, and $A$ is invertible, then
\[
\|A B\|=\|A^{1/2} B A^{1/2}\|=\inf\{\lambda: B\preceq \lambda A^{-1}\}. 
\]
\end{lem}
\begin{proof}
$\|A B\|=\|A^{1/2} B A^{1/2}\|$ comes as conjugacy preserves eigenvalues, and $\|A^{1/2} B A^{1/2}\|=\inf\{\lambda: B\preceq \lambda A^{-1}\}$ is obvious. 
\end{proof}
\begin{lem}
\label{lem:trace}
Let $B\in PSD$, then $\text{tr}(AB)\leq \|A\|\text{tr}(B)$.
\end{lem}
\begin{proof}
Suppose the eigenvalue decomposition of $B$ is $\Psi D\Psi^T$. Then we note that
\[
\text{tr}(AB)=\text{tr}(A\Psi D\Psi^T)=\text{tr}(\Psi^T A\Psi D),\quad \|A\|\text{tr}(B)=\|\Psi^T A\Psi\| \text{tr}(D). 
\]
So we can assume $B$ is a diagonal matrix. Then
\[
\text{tr}(AB)=\sum_{i=1}^d A_{i,i} B_{i,i}\leq \|A\| \sum B_{i,i}=\|A\|\text{tr}(B). 
\]
\end{proof}
\section{Convergence to the unique stationary solution}
One of the remarkable property of Kalman filter covariance is that it converges to a unique stationary solution to the associated Riccati equation, assuming the system coefficients are stationary, and weak observability and controllability. 
\label{sec:stationary}
\begin{thm}[Bougerol 93]
\label{thm:Bou93}
Suppose that $(A_n, B_n, H_n, \Sigma_n)$ is an ergodic stationary sequence. Define the observability and controllability Gramian as follows:
\[
\mathcal{O}_n=\sum_{k=1}^n A_{k,1}^TH_k^T\sigma_k^{-1}H_kA_{k,1},\quad \mathcal{C}_n=\sum_{k=1}^n A_{n,k+1}^T\Sigma_kA_{n,k+1},\quad A_{k,j}=A_k A_{k-1}\cdots A_j. 
\]
Suppose the system\eqref{sys:random}  is weakly observable and controllable, that is there is an $n$ such that 
\[
\Prob(\det (\mathcal{O}_n)\neq 0, \det (\mathcal{C}_n)\neq 0)>0.
\]  
Suppose also the following random variables are integrable,
\[
\log\log^+\|A_1\|,\quad \log\log^+ \|A_1^{-1}\|,\quad \log\log^+\|\Sigma_1\|,\quad \log\log^+\|H_1^TH_1\|.
\]
where $\log^+x=\max\{0,\log x\}$. Then there is a stationary PD sequence $\Rtilde_n$ that follows
\[
\Rtilde_{n+1}=\Kalman(A_n \Rtilde_n A_n^T+\Sigma_n).
\]
And for the covariance matrix $R_n$ of another Kalman filter started with an initial value $R_0$, will converge to $R^s_n$ asymptotically: $\lim\sup_{n\to \infty} \frac{1}{n}\delta (R_n, \Rtilde_n)\leq \alpha$. Here $\alpha$ is a negative constant, and $\delta$ defines a Riemannian distance on $S^+_d$ by 
\[
\delta (P,Q)=\sqrt{\sum_{i=1}^d\log^2 \lambda_i},\quad \lambda_i\text{ are eigenvalues of }PQ^{-1}.  
\]
\end{thm}
One simple and useful fact is that, if the original system \eqref{sys:random} meets the requirement of Theorem \ref{thm:Bou93}, then so does the inflated systems \eqref{sys:inflatedDRKF} and \eqref{sys:inflated}. To see this, one need to write down the corresponding observability and controllability Gramians $\mathcal{O}'_{n}$ and $\mathcal{C}'_{n}$. 

In the general scenarios, it's straightforward to verify that the Gramians of  reference system \eqref{sys:inflated} are larger than the ones of \eqref{sys:random}
\[
\mathcal{O}'_{n}=\sum_{k=1}^n A_{k,1}^{'T}H_k^T\sigma_k^{-1}H_kA^T_{k,1}\succeq \mathcal{O}_{n},\quad \mathcal{C}'_{n}=\sum_{k=1}^n A_{n,k+1}^{'T}\Sigma'_kA'_{n,k+1}\succeq \mathcal{C}_{n}. 
\]
Since these Gramian matrices are PSD matrices, $\mathcal{O}_{n}$ and $\mathcal{C}_{n}$ are nonsingular indicate that $\mathcal{O}'_{n}$ and $\mathcal{C}'_{n}$ are nonsingular. Also note 
\[
\log \|A_1'\|=\frac{1}{2}\log \Rtilde+ \log\|A_1\|,
\]
moreover, 
\begin{align*}
 \log^+ \|\Sigma'_1\|&\leq \log \Rtilde+\log^+(\|\Sigma_n\|+\|A_n D_S A_n^T\|) \\
 &\leq \log \Rtilde+\max\{\log^+(2\|\Sigma_n\|),\log^+(2\|A_n D_S A_n^T\|)\} \\
 &\leq \log \Rtilde+\log 2+\log^+ \|\Sigma_n\|+2\log^+\|A_n\|+\log^+ \|D_S\|
\end{align*}
so the integrability condition holds naturally.

When the system  had dynamical two-scale decoupling \eqref{eqn:blockdiag}, it is easy to see that $A_{k,j}$ has a block-diagonal structure, and so do the Gramians $\mathcal{C}_n$ and $\mathcal{O}_n$. It is also easy to verify the controllability Gramian of \eqref{sys:inflatedDRKF} satisfies 
\[
\mathcal{C}'_n=\sum_{k=1}^n A_{n,k+1}^{'T}\Sigma^L_kA'_{n,k+1}\succeq \sum_{k=1}^n A_{n,k+1}^{T}\Sigma^L_kA^L_{n,k+1} =\bfP_L \mathcal{C}_n\bfP_L.
\] As for the observability Gramian of \eqref{sys:inflatedDRKF}, notice that $\sigma_k^{-1}$ is invertible, so there is a constant $D_n$ such that $D_n \sigma_k^{-1}\preceq [\sigma_k^L]^{-1}$, then it is straightforward to verify that 
\[
D_n\mathcal{O}'_n=D_n\sum_{k=1}^n A_{k,1}^{'T}H_k^T\sigma_k^{-1}H_kA^{'T}_{k,1}\succeq \bfP_L\mathcal{O}_{n}\bfP_L. 
\]
This shows the weak observability and controllability of \eqref{sys:random} implies the ones of \eqref{sys:inflatedDRKF}. The integrability

%\end{proof}
\bibliographystyle{unsrt}
\bibliography{lowkalman}
\end{document}